\documentclass[a4paper,11pt]{amsart}

\usepackage[latin1]{inputenc}
\usepackage[english]{babel}
\usepackage{hyperref,tikz}

\usetikzlibrary{3d}
\usetikzlibrary{calc}
\usetikzlibrary{math}

\pgfmathsetmacro\xx{1/sqrt(2)}
\pgfmathsetmacro\xy{1/sqrt(6)}
\pgfmathsetmacro\zy{sqrt(2/3)}

\tikzmath%
{%
  function crossx(\mx,\my,\mz,\nx,\ny,\nz)
  {
    \pxx = \my*\nz-\mz*\ny;
    \pyy = \mz*\nx-\mx*\nz;
    \pzz = \mx*\ny-\my*\nx;
    return {\pxx/sqrt(\pxx*\pxx+\pyy*\pyy+\pzz*\pzz)};
  };
  function crossy(\mx,\my,\mz,\nx,\ny,\nz)
  {
    \pxx = \my*\nz-\mz*\ny;
    \pyy = \mz*\nx-\mx*\nz;
    \pzz = \mx*\ny-\my*\nx;
    return {\pyy/sqrt(\pxx*\pxx+\pyy*\pyy+\pzz*\pzz)};
  };
  function crossz(\mx,\my,\mz,\nx,\ny,\nz)
  {
    \pxx = \my*\nz-\mz*\ny;
    \pyy = \mz*\nx-\mx*\nz;
    \pzz = \mx*\ny-\my*\nx;
    return {\pzz/sqrt(\pxx*\pxx+\pyy*\pyy+\pzz*\pzz)};
  };
}

\newcommand{\greatcircle}[6] 
{%
  \coordinate (P)  at (#1,#2,#3);                  
  \coordinate (N)  at ($(0,0,0)!#6*1.25cm!(P)$);   
  \coordinate (S)  at ($-1*(N)$);                  
  \coordinate (E)  at ($(0,0,0)!-1.25cm!270:(P)$); 
  \coordinate (W)  at ($-1*(E)$);                  
  \coordinate (NW) at ($(N)+(W)$);
  \coordinate (NE) at ($(N)+(E)$);
  \coordinate (SW) at ($(S)+(W)$);
  \coordinate (SE) at ($(S)+(E)$);
  \pgfmathsetmacro\ptheta{atan(#2/#1)} 
  \pgfmathsetmacro\pphi  {#5*acos(#3)} 
  \begin{scope}
    \clip (W) -- (SW) -- (SE) -- (E) -- cycle;
    \draw[rotate around z=\ptheta,rotate around y=\pphi,%
          canvas is xy plane at z=0,#4] (0,0) circle (1);
  \end{scope}
  \begin{scope}
    \clip (W) -- (NW) -- (NE) -- (E) -- cycle;
    \draw[rotate around z=\ptheta,rotate around y=\pphi,%
          canvas is xy plane at z=0,#4,densely dotted] (0,0) circle (1);
  \end{scope}
}

\usepackage[T1]{fontenc}
\usepackage{amsmath,bbold}
\usepackage{amsfonts}
\usepackage{amssymb}
\usepackage{amsthm,eepic}
\usepackage{mathrsfs}
\usepackage{enumitem}
\usepackage{graphicx}
\usepackage{footnote}
\usepackage{hyperref,tikz}
\usetikzlibrary{quotes,angles}
\hypersetup{colorlinks=true,    linkcolor=blue,
citecolor=red,       filecolor=BrickRed,       urlcolor=darkgreen
}

\renewcommand{\geq}{\geqslant}
\renewcommand{\leq}{\leqslant}
\newtheorem{thm}{Theorem}
\newtheorem{prop}{Proposition}
\newtheorem{rem}{Remark}
\newtheorem{cor}[prop]{Corollary}
\newtheorem{lem}[prop]{Lemma}

\newcommand{\be}{\begin{equation}}
\newcommand{\ee}{\end{equation}}

\usepackage{pgfplots}

\pgfplotsset{compat=newest}

\colorlet{darkgreen}{green!50!black}
\definecolor{darkseagreen}{rgb}{0.56, 0.74, 0.56}
\definecolor{lightcyan}{rgb}{0.88, 1.0, 1.0}
\definecolor{lightblue}{rgb}{0.68, 0.85, 0.9}
\definecolor{palecerulean}{rgb}{0.61, 0.77, 0.89}
\definecolor{lgreen} {RGB}{180,210,100}
\definecolor{dblue}  {RGB}{20,66,129}
\definecolor{ddblue} {RGB}{11,36,69}
\definecolor{lred}   {RGB}{220,0,0}
\definecolor{nred}   {RGB}{224,0,0}
\definecolor{norange}{RGB}{230,120,20}
\definecolor{nyellow}{RGB}{255,221,0}
\definecolor{ngreen} {RGB}{98,158,31}
\definecolor{dgreen} {RGB}{78,138,21}
\definecolor{nblue}  {RGB}{28,130,185}
\definecolor{jblue}  {RGB}{20,50,100}
\definecolor{Apricot} {RGB}{255, 170, 123} 
\definecolor{dpurple}  {RGB}{53,21,93}

\def\R{{\mathbb R}}

\def\N{{\mathbb N}}

\def\1{\mathbb{1}}

\def\eps{{\epsilon}}

\def\B{{\mathbb B}}

\DeclareMathOperator{\dom}{dom}
\DeclareMathOperator{\arccot}{arccot}

\def\R{\mathbb{R}}

\renewcommand{\epsilon}{\varepsilon}

\usepackage{color}
\definecolor{darkgreen}{rgb}{0,0.4,0}
\definecolor{MyDarkBlue}{rgb}{0,0.08,0.50}
\definecolor{BrickRed}{rgb}{0.65,0.08,0}

\newcommand{\Mcal}{\mathcal{M}}
\renewcommand{\cot}{\mathrm{cot}\,}
\newcommand{\ep}{\varepsilon}
\newcommand{\Jac}{\mathrm{Jac}\,}
\newcommand{\spec}{\mathrm{Spec}\,}

\newcommand{\Abb}{\mathbb{A}}

\newcommand{\Gbb}{\mathbb{G}}
\newcommand{\Ecal}{\mathcal{E}}
\newcommand{\Bcal}{\mathcal{B}}
\newcommand{\Dcal}{\mathcal{D}}
\newcommand{\Hcal}{\mathcal{H}}

\newcommand{\alp}{\alpha}
\newcommand{\dis}{\displaystyle}
\newcommand{\pid}{\frac{\pi}{2}}

\title[On the rationality of 3D lattice walks confined to an octant]{Lattice walks confined to an octant in dimension 3:\\(non-)rationality of the second critical exponent}

\author{Luc Hillairet}
\address{Institut Denis Poisson, Universit\'e de Tours et Universit\'e d'Orl\'eans, Parc de Grandmont, 37200 Tours, France}
\email{luc.hillairet@univ-orleans.fr}

\author{Helen Jenne}
\address{CNRS \and Institut Denis Poisson, Universit\'e de Tours et Universit\'e d'Orl\'eans, Parc de Grandmont, 37200 Tours, France}
\email{helen.jenne@lmpt.univ-tours.fr}

\author{Kilian Raschel}
\address{CNRS \and Laboratoire Angevin de Recherche en Math\'ematiques, Universit\'e d'Angers, 2 boulevard Lavoisier, 49000 Angers, France}
\email{raschel@math.cnrs.fr}
\thanks{This project has received funding from the European Research Council (ERC) under the European Union's Horizon 2020 research and innovation programme under the Grant Agreement No.\ 759702.}

\keywords{Lattice walks; Asymptotic enumeration; D-finite series; Heat kernel; Perturbation theory}
\subjclass[2010]{Primary ; Secondary }

\date{\today}

\begin{document}

\begin{abstract}
In the field of enumeration of walks in cones, it is known how to compute asymptotically the number of excursions (finite paths in the cone with fixed length, starting and ending points, using jumps from a given step set). As it turns out, the associated critical exponent is related to the eigenvalues of a certain Dirichlet problem on a spherical domain. An important underlying question is to decide whether this asymptotic exponent is a (non-)rational number, as this has important consequences on the algebraic nature of the associated generating function. In this paper, we ask  whether such an excursion sequence might admit an asymptotic expansion with a first rational exponent and a second non-rational exponent. While the current state of the art does not give any access to such many-term expansions, we look at the associated continuous problem, involving Brownian motion in cones. Our main result is to prove that in dimension three, there exists a cone such that the heat kernel (the continuous analogue of the excursion sequence) has the desired rational/non-rational asymptotic property. Our techniques come from spectral theory and perturbation theory. More specifically, our main tool is a new Hadamard formula, which has an independent interest and allows us to compute the derivative of eigenvalues of spherical triangles along infinitesimal variations of the angles.
\end{abstract}

\maketitle

\section{Introduction}

\subsection*{The model and our main question}
A lattice walk is a sequence of points $P_0, P_1,\ldots ,P_n$ of $\mathbb Z^d$, $d\geq 1$. The points $P_0$ and $P_n$ are its starting and end points, respectively, the consecutive differences $P_{i+1}-P_i$ its steps, and $n$ is its length. Given a set $\mathcal S\subset \mathbb Z^d$, called the step set, a set $C\subset \mathbb Z^d$ called the domain (which in this paper will systematically be a cone), and elements $P$ and $Q$ of $C$, we are interested in the number $e(P,Q; n)$ of walks (or excursions) of length $n$ that start at $P$, have all their steps in $\mathcal S$, have all their points in $C$, and end at $Q$.
In the present note, the main problem we would like to address is the following: does there exist a walk model (i.e., a step set and a cone in $\mathbb R^d$) such that as $n\to\infty$, one has the asymptotics
\begin{equation}
\label{eq:two-term-asymp}
   e(P,Q;n) = \rho^n \cdot \bigl(K_1\cdot  n^{\alpha_1} + K_2\cdot  n^{\alpha_2} +\cdots +  K_p\cdot  n^{\alpha_p}  +o(n^{\alpha_p}) \bigl),
\end{equation}
with some exponential growth $\rho>0$ and critical exponents such that $\alpha_1,\ldots,\alpha_{p-1}\in\mathbb Q$ and $\alpha_p\notin\mathbb Q$? (The constants $K_1,\ldots ,K_p$ are assumed to be non-zero.)
We now present the context and explain our motivations to look at this particular problem.

\subsection*{Asymptotics of the excursion sequence and relation to D-finiteness}
Is there a simple formula for $e(P,Q;n)$ in terms of the coordinates of $P,Q$ and the length $n$ of the walk? If not, can we at least say something about the asymptotic behaviour \eqref{eq:two-term-asymp} of these numbers as $n$ goes to infinity? A first step towards answering these questions can be done by considering the excursion generating function
\begin{equation}
\label{eq:def_generating_function}
   e_{P,Q}(t)=\sum_{n\geq 0}e(P,Q;n) t^n\in\mathbb Q[[t]]
\end{equation}
that is associated with these numbers and determining whether it is algebraic, or, if not, whether it is at least D-finite. Recall that a series is D-finite if it satisfies a non-trivial linear differential equation with polynomial coefficients. Knowing that a given series is D-finite not only implies nice computational properties of its coefficients, but also allows us to classify the combinatorial model according to the complexity of the underlying generating function. There has recently been a dense literature around the above questions, in relation with the probabilistic model of random walks in cones. 

As it turns out, there is a strong relation between D-finiteness of a given series and the asymptotic behavior of its coefficients. For example, the following statement (recalled in \cite[Thm~3]{BoRaSa-14}) is a consequence of results by Andr\'e, Chudnovski and Katz:
\begin{lem}
\label{thm:AnChKa}
Let $(e(n))_{n\geq0}$ be an integer-valued sequence whose $n$-th term behaves~asymptotically like
\begin{equation}
\label{eq:one-term-asymp}
   e(n)\sim K\cdot \rho^n \cdot n^\alpha
\end{equation}
for some real positive constants $K$ and $\rho$. If the singular exponent $\alpha$ is irrational, then the generating function $e(t) = \sum\limits_{n=0}^\infty e(n) t^n$ is not D-finite.
\end{lem}

Given the above result, it is natural to ask whether one may compute and study the rationality of the critical exponent $\alpha$ in the asymptotics \eqref{eq:one-term-asymp} of the excursion sequence $e(P,Q;n)$ (equivalently, the dominant term in the asymptotics \eqref{eq:two-term-asymp}).

In dimension $1$, the combinatorial model of walks in cones reduces to that of walks confined to the positive half-line, as studied e.g.\ in \cite{BaFl-02}. In this context, it is well known that only simple exponents appear in the dominant asymptotics, namely $\alpha=0$, $-\frac{1}{2}$ or $-\frac{3}{2}$ (depending on the drift of the model), and their translations by integers in the complete asymptotic expansion \eqref{eq:two-term-asymp}. Accordingly, there is nothing to say from the perspective of the rationality of asymptotic exponents. We remark that these simple exponents are deduced from the algebraicity of the associated generating function \eqref{eq:def_generating_function} and classical transfer theorems (singularity analysis).

Given a cone in higher dimension $d\geq 2$, the generating function \eqref{eq:def_generating_function} is in general not algebraic (and even non-D-finite, see \cite{KuRa-12,BoRaSa-14} for the case of the quarter plane in dimension $2$), and the first problem is to access the critical exponent $\alpha$. This result is obtained by Denisov and Wachtel \cite{DeWa-15}: for a large class of cones in arbitrary dimension, they derive the one-term asymptotics \eqref{eq:one-term-asymp} for the excursion sequence $e(P,Q;n)$. In particular, they show that \cite[Eq.~(12)]{DeWa-15}
\begin{equation}
\label{eq:formula_alpha_DW}
   \alpha=-\sqrt{\lambda_1+(d/2-1)^2}-1,
\end{equation}
where $d$ is the dimension and $\lambda_1$ is interpreted as the principal Dirichlet eigenvalue for the Laplace-Beltrami operator on the subdomain of the sphere $\mathbb S^{d-1}$ given by 
\begin{equation}
\label{eq:def_domain}
   (LC)\cap\mathbb S^{d-1},
\end{equation}
with $C$ being the domain of confinement (typically an orthant $\mathbb R_+^d$) and $L$ a linear application, which depends on the model. One should not be surprised by the presence of the linear transform $L$ in \eqref{eq:def_domain}: as a matter of comparison, the classical central limit theorem for random walks in $\mathbb R^d$ involves the drift and the covariance matrix of the process, so as to put the random walk in the domain of attraction of a standard Brownian motion (here, standard means without drift and with identity covariance matrix). Similarly, the application $L$ above appears so as to take into account the drift and the covariance matrix of the combinatorial model under consideration.

Some key ingredients in Denisov and Wachtel's proof are a coupling of random walk by Brownian motion and then a use of older results in the probabilistic literature on exit times for Brownian motion \cite{De-87,BaSm-97,Va-99} (the eigenvalue $\lambda_1$ already appearing in the study of Brownian motion in cones).

Accordingly, all the complexity of the excursion (one-term) asymptotics \eqref{eq:one-term-asymp} is contained in the principal eigenvalue $\lambda_1$.

\subsection*{Dimension $2$}
Regarding the combinatorial model of walks in the quarter plane, the domain \eqref{eq:def_domain} simply becomes an arc of circle, see Figure~\ref{fig:step_sets_2D} for a few examples. More precisely, if the walk is driftless and has identity covariance matrix, then $L$ is just the identity and \eqref{eq:def_domain} is a quarter of circle. For other walk models, using the expression of the linear transform $L$, the arc has opening $\beta\in(0,\pi)$, which one may express as $\arccos(-r)$, where $r$ is an algebraic number which is easily computed from the model; see \cite{BoRaSa-14} for more details.

As it turns out, the principal eigenvalue (and in fact the whole spectrum) of arcs of circles is known. More precisely, if the cone has opening $\beta$, then $\lambda_1=(\frac{\pi}{\beta})^2$, and more generally the $j$-th eigenvalue is given by $\lambda_j=(j\frac{\pi}{\beta})^2$. Consequently, using \eqref{eq:formula_alpha_DW}, one deduces that the asymptotic exponent $\alpha$ is known and is equal to
\begin{equation*}
   \alpha = -\frac{\pi}{\beta}-1 = -\frac{\pi}{\arccos (-r)}-1.
\end{equation*}
For instance, for the model on the left on Figure \ref{fig:step_sets_2D} (called a scarecrow in \cite{BoRaSa-14}), one has $r=1/4$ and thus $\alpha=-\frac{\pi}{\arccos (-1/4)}-1$, which can be proved to be non-rational \cite{BoRaSa-14}.

Following this approach, the authors of \cite{BoRaSa-14} obtain that for a list of $51$ (unweighted, having infinite group and small steps) models, $\alpha$ is non-rational, and so, using Lemma~\ref{thm:AnChKa}, these $51$ models admit non-D-finite generating functions. In the context of unweighted quadrant lattice walks, it is remarkable that the converse statement is also true: in other words, the generating function \eqref{eq:def_generating_function} of the $74$ non-singular, unweighted quadrant lattice walks is D-finite if and only if the principal exponent is rational.

This equivalence (between D-finiteness of the generating function and rationality of the critical exponent) is a priori not true in general: the authors of \cite{BoBMMe-18} construct several models (one of them is represented on Figure \ref{fig:step_sets_2D}, right) for which $\alpha$ is rational but the generating function is conjectured to be non-D-finite. See Table 2 in \cite{BoBMMe-18} for more examples.

With this in mind, our question in dimension $2$ would be to see whether there exist quadrant walk models such that the associated excursion sequence admits the asymptotics
\eqref{eq:two-term-asymp}, with $\alpha_1,\ldots,\alpha_{p-1}\in \mathbb Q$ and $\alpha_p\notin \mathbb Q$. Such a statement would also lead to non-D-finiteness results, by a generalization of Lemma \ref{thm:AnChKa} to many-term asymptotic expansions. See in particular the works \cite{FiRi-14,FiRi-19}, where this generalization is mentioned.

As we will explain later, we conjecture that the above rationality/non-rationality phenomenon does not occur in dimension~$2$.

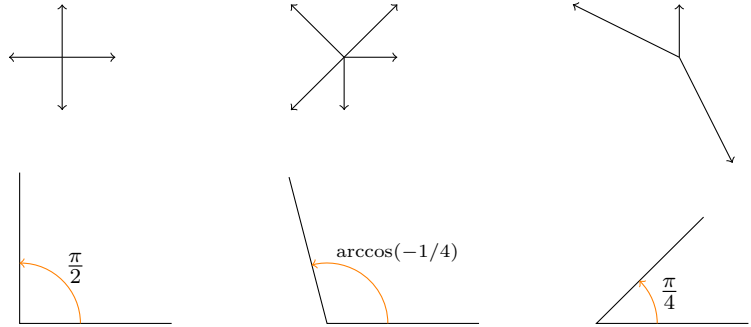
\begin{figure}
\begin{center}
\begin{tikzpicture}[scale=.7] 
    \draw[->,white] (1,2) -- (0,-2);
    \draw[->,white] (1,-2) -- (0,2);
    \draw[->] (0,0) -- (1,0);
    \draw[->] (0,0) -- (-1,0);
    \draw[->] (0,0) -- (0,-1);
    \draw[->] (0,0) -- (0,1);
   \end{tikzpicture}\qquad\qquad\qquad
\begin{tikzpicture}[scale=.7] 
    \draw[->,white] (1,2) -- (0,-2);
    \draw[->,white] (1,-2) -- (0,2);
    \draw[->] (0,0) -- (1,1);
    \draw[->] (0,0) -- (1,0);
    \draw[->] (0,0) -- (0,-1);
    \draw[->] (0,0) -- (-1,1);
    \draw[->] (0,0) -- (-1,-1);
   \end{tikzpicture}\qquad\qquad\qquad
\begin{tikzpicture}[scale=.7] 
  \draw[->,white] (1.5,2) -- (-2,-2);
    \draw[->,white] (1.5,-2) -- (-2,2);
        \draw[->] (0,0) -- (0,1);
    \draw[->] (0,0) -- (-2,1);
    \draw[->] (0,0) -- (1,-2);
   \end{tikzpicture}
  \end{center}
  \begin{center}
  \begin{tikzpicture}
  \draw
    (2,0) coordinate (a) 
    -- (0,0) coordinate (b) 
    -- (0,2) coordinate (c) 
    pic["$\frac{\pi}{2}$", draw=orange, ->, angle eccentricity=1.3, angle radius=0.8cm]
    {angle=a--b--c};
\end{tikzpicture}\qquad\qquad
\begin{tikzpicture}
  \draw
    (2,0) coordinate (a) 
    -- (0,0) coordinate (b) 
    -- (-0.5,1.94) coordinate (c) 
    pic["\quad\tiny{$\arccos(-1/4)$}", draw=orange, ->, angle eccentricity=1.5, angle radius=0.8cm]
    {angle=a--b--c};
\end{tikzpicture}\qquad\qquad
\begin{tikzpicture}
  \draw
    (2,0) coordinate (a) 
    -- (0,0) coordinate (b) 
    -- (1.41,1.41) coordinate (c) 
    pic["$\frac{\pi}{4}$", draw=orange, ->, angle eccentricity=1.3, angle radius=0.8cm]
    {angle=a--b--c};
\end{tikzpicture}
\end{center}
  \caption{Three examples of walks in dimension $2$ confined to the quarter plane $C=\mathbb R_+^2$, with the associated domain $(LC)\cap\mathbb S^{1}$ as in \eqref{eq:def_domain}. Their critical exponent $\alpha$ in \eqref{eq:one-term-asymp} is as follows: on the left, $\alpha=-3$; for the second model $\alpha=-\frac{\pi}{\arccos(-1/4)}-1\notin \mathbb Q$, see \cite{BoRaSa-14}; on the right $\alpha=-5$, see \cite{BoBMMe-18}.}
  \label{fig:step_sets_2D}
\end{figure}

\subsection*{Dimension $3$}
We now explore the case of dimension $3$. First, the domain \eqref{eq:def_domain} to consider is the trace on the sphere $\mathbb S^2$ of $L\mathbb R_+^3$, which by construction is a spherical triangle, see Figure~\ref{fig:tilings} for a few examples. In other words, in dimension $3$, one has to understand the principal eigenvalue $\lambda_1$ of spherical triangles. This connection between three-dimensional positive lattice walks and spherical triangles has been studied in \cite{BoPeRaTr-20}, see also \cite{RaTr-09} in relation with a Brownian pursuit problem.

While in dimension $2$, it was possible to compute the whole spectrum for the Laplace-Beltrami problem with Dirichlet conditions on the domain \eqref{eq:def_domain}, and in addition we had nice formulas for all eigenvalues and eigenfunctions (recall that $\lambda_j=(j\frac{\pi}{\beta})^2$ in the planar case), this is no longer the case in dimension $3$. More precisely, given a generic spherical triangle, it is in general impossible to compute in closed form any of its eigenvalues. To summarize, up to our knowledge, there are only two kinds of exceptional spherical triangles which admit eigenvalues in closed form: \begin{itemize}
   \item Spherical triangles corresponding to tilings of the sphere \cite{Be-93}. Notice that tilings do not all lead to an explicit spectrum: for instance, the one on the right on Figure~\ref{fig:tilings} (called the tetrahedral tiling) cannot be solved in an explicit manner, as it does not admit the right parity). Specifically, the angles of these triangles should take one of the following values: $(\frac{\pi}{2},\frac{\pi}{3},\frac{\pi}{3})$, $(\frac{\pi}{2},\frac{\pi}{3},\frac{\pi}{4})$, $(\frac{\pi}{2},\frac{\pi}{3},\frac{\pi}{5})$ or $(\frac{\pi}{2},\frac{\pi}{2},\frac{\pi}{r})$, with some integer $n\geq 2$.
   \item Arbitrary birectangular triangles, i.e., triangles admitting two right angles $\frac{\pi}{2}$ and one arbitrary angle $\beta\in(0,2\pi)$; see \cite{Wa-74,WaKe-77,SeWeZh-21}. 
\end{itemize}
In this three-dimensional context, our question takes the following form: does there exist an octant walk model such that the associated excursion sequence admits the asymptotics \eqref{eq:two-term-asymp} with $\alpha_1,\ldots,\alpha_{p-1}\in \mathbb Q$ and $\alpha_p\notin \mathbb Q$?

\begin{figure}
\begin{center}
\includegraphics[width=0.25\textwidth]{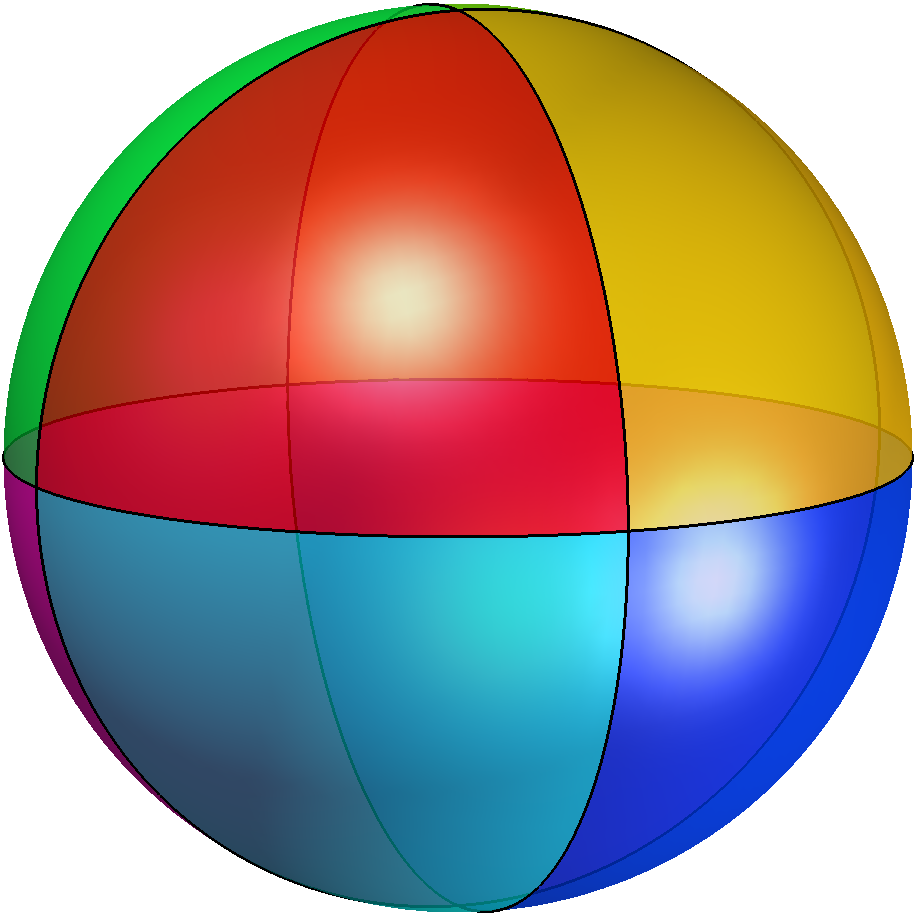}\qquad\qquad\qquad
\includegraphics[width=0.25\textwidth]{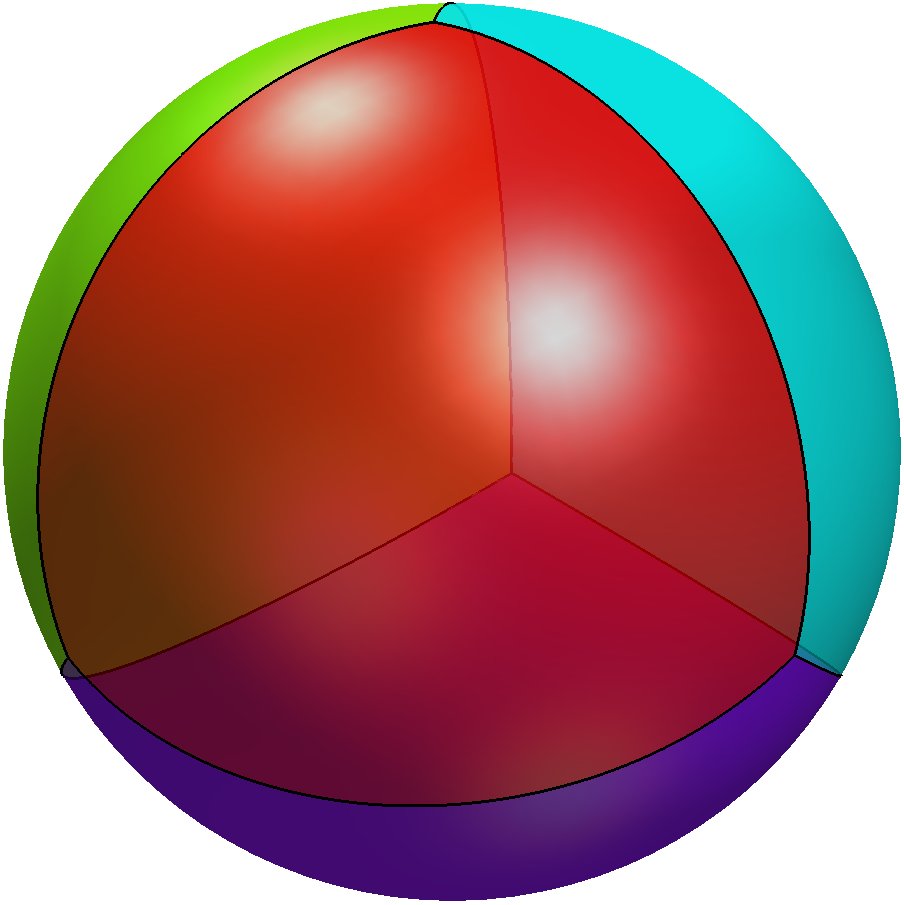}
\caption{Two tilings of the sphere by triangles. The left one corresponds to the simple walk, with jumps $\{(\pm1,0,0),(0,\pm1,0),(0,0,\pm1)\}$. The tiling on the right is associated to the following model, known as 3D Kreweras model: $\{(-1,0,0),(0,-1,0),(0,0,-1),(1,1,1)\}$. The rationality of its critical exponent $\alpha$ is still unknown.}
\label{fig:tilings}
\end{center}
\end{figure}

\subsection*{The heat kernel of cones}
To answer our main question, one intrinsic difficulty is to know a many-term asymptotic expansion of the form of \eqref{eq:two-term-asymp}. And indeed, such asymptotics are not available in the literature in general (except in a few very particular cases, which are the simplest cases, so precisely those with a complete asymptotic expansions with rational exponents, see e.g.\ \cite{BMMi-10}).

As a consequence, in order to progress on our question, we will reason by analogy between the discrete setting (random walk) and the continuous setting (Brownian motion), and we will solve the analogous question in the Brownian framework. 

First of all, the quantity analogous to the number of excursions $e(P,Q;n)$ is called the (continuous) heat kernel of the cone, which, as we shall see, admits an expression in closed-form \eqref{eq:explicit_expression_heat_kernel} and explicit complete asymptotic expansions.

The heat kernel $p^C(x,y;t)$ of a cone (and actually of any domain) $C$ admits the following probabilistic interpretation: it is the probability density function of the transition probability kernel
\begin{equation}
\label{eq:def_heat_kernel}
   p^C(x,y;t) dy =\mathbb P_x(B_t\in dy,\tau_C>t),
\end{equation}
where the Brownian motion is denoted by $B_t$ and $\tau_C$ is the first exit time from the cone $C$, that is, $\tau_C = \inf\{t>0 : B_t\notin C\}$. Letting $0<\lambda_1<\lambda_2\leq \lambda_3\leq \cdots$ denote the eigenvalues of the Laplace-Beltrami operator with Dirichlet conditions on the domain $C\cap\mathbb S^{d-1}$, its explicit expression is given by \cite{BaSm-97}
\begin{equation}
\label{eq:explicit_expression_heat_kernel}
   p^C(x,y;t)=\frac{\exp\left(-\frac{\vert x\vert^2+\vert y\vert^2}{2t}\right)}{t(\vert x\vert \cdot \vert y\vert)^{d/2-1}} \sum_{j=1}^{\infty} I_{\alpha_j}\left(\frac{\vert x\vert \cdot \vert y\vert}{t}\right) m_j\left(\frac{x}{\vert x\vert}\right)m_j\left(\frac{y}{\vert y\vert}\right),
\end{equation}
where $\alpha_j=\sqrt{\lambda_j+(d/2-1)^2}$, $m_j$ is the associated normalized eigenfunction, and $I_\nu$ is the modified Bessel function of order $\nu$, which admits the expression
\begin{equation*}
   I_\nu(x) = \sum_{m=0}^\infty \frac{1}{m!\Gamma(\nu+m+1)} \left(\frac{x}{2}\right)^{\nu +2m}.
\end{equation*}

The following result may be found in \cite[Thm~2.3]{ChFuRa-20}:
\begin{lem}
\label{lem:heat_kernel_asymptotics}
For any dimension $d\geq 1$ and any cone regular enough, the heat kernel $p^C(x,y;t)$ in \eqref{eq:def_heat_kernel} admits a complete asymptotic expansion of the form
\begin{equation}
\label{eq:many-term-asymp_heat-kernel}
   p^C(x,y;t) = K_1\cdot  t^{-\alpha_1} + K_2\cdot  t^{-\alpha_2} +\cdots +  K_p\cdot  t^{-\alpha_p}  +o(t^{-\alpha_p}),
\end{equation}
where
\begin{itemize}
   \item the order $p$ of the expansion is arbitrary large;
   \item the constants $K_i$ depend on $x$ and $y$, i.e., $K_i=K_i(x,y)$;
   \item the exponents $\alpha_i$ are independent of $x$ and $y$, and $\alpha_1<\alpha_2<\cdots<\alpha_p$;
   \item $\alpha_i=\sqrt{\lambda_j+(d/2-1)^2}+k$, where $\lambda_j$ is an eigenvalue and $k$ is a positive integer.
\end{itemize}
\end{lem}

In this new (and last!) setting, our question may be formulated as follows: is it possible in the asymptotics \eqref{eq:many-term-asymp_heat-kernel} to have first rational exponents $\alpha_1,\ldots ,\alpha_{p-1}$ and a non-rational $\alpha_p$? 

\subsection*{Statements of our main results}
As the following proposition establishes, our question is easily solved in dimension $2$, and the answer happens to be negative.

\begin{prop}
In dimension $d=2$, the exponents $\alpha_i$ appearing in the asymptotics \eqref{eq:many-term-asymp_heat-kernel} of the heat kernel are simultaneously all rational or non-rational.
\end{prop}
\begin{proof}
Using that $d=2$ and $\lambda_j=(j\frac{\pi}{\beta})^2$, it follows from Lemma~\ref{lem:heat_kernel_asymptotics} that the exponents $\alpha_i$ may be expressed as $j\frac{\pi}{\beta}+k$, where $j$ and $k$ are positive integers. Clearly, when $j$ and $k$ vary, these numbers are either all rational or all non-rational. 
\end{proof}

Accordingly, we also conjecture that we cannot construct any discrete model having this rationality/non-rationality property (with sufficiently many moment conditions). 

Although we shall not elaborate on this here, we would like to mention that, based on the above two-dimensional result, it should be easy to give an example to our rationality/non-rationality phenomenon in dimension $4$, seeing $\mathbb R^4$ as a product of two planes and defining on each plane a different model, one with $\frac{\pi}{\beta_1}\in\mathbb Q$ and the second one with $\frac{\pi}{\beta_2}\in\mathbb Q$. We thank Andrew Elvey Price for this suggestion.

So we have to move to dimension $3$. Our main theorem in this paper is the following:
\begin{thm}
\label{thm:main_formulation_comb}
There exists a 3D cone such that the heat kernel admits the asymptotics \eqref{eq:many-term-asymp_heat-kernel}, with first rational exponents $\alpha_1,\ldots ,\alpha_{p-1}$ and then a non-rational exponent $\alpha_p$.
\end{thm}
Theorem \ref{thm:main_formulation_comb} is a rather direct consequence of the following result:

\begin{thm}
\label{thm:main_formulation_spectral}
There exists $t_0>0$ and a real analytic function $b$ defined on $(-t_0,t_0)$, such that the one parameter
family of triangles $(T_t)_{t\in (-t_0,t_0)}$ that have one side of length
$\frac{\pi}{2}$ and adjacent angles with values 
\begin{equation*}
   \frac{\pi}{2}+t,\frac{\pi}{2}+b(t),
\end{equation*}
satisfies, for the Dirichlet Laplace operator, 
\begin{itemize}
\item the first eigenvalue $ \lambda_1(t)$ of $T_t$ is constant: $\forall t\in (-t_0,t_0),~\lambda_1(t)=12,$
\item the second eigenvalue $\lambda_2(t)$ admits the first order approximation:
  \begin{equation*}
    \lambda_2(t)\,=\,30-22\sqrt{3}\cdot |t| \,+\,o(t).
  \end{equation*}
\end{itemize} 
\end{thm}

\subsection*{Acknowledgments} 
The last author would like to thank Alin Bostan for very interesting discussions related to the rationality of asymptotic exponents and the relation to non-D-finiteness.

\section{The spectrum of spherical triangles}
We prove Theorem \ref{thm:main_formulation_spectral} by studying the first eigenvalues
as functions on the set $\Mcal$ of spherical triangles with one side of length $\frac{\pi}{2}$. We first show that the level sets of the
first eigenvalue $\lambda_1$ are analytic curves in $\Mcal$. Denote by $T_*$ the equirectangle triangle, see Figure~\ref{fig:ex1} (left). Restricting to the curve on which the first eigenvalue is constant and equal to $12=\lambda_1(T_*)$,
we compute the derivatives of the second and third eigenvalue branches at $T_*$.
Since the latter derivatives do not vanish, the theorem will be proved.

This strategy of proof relies heavily on analytic perturbation theory (see \cite{Kato}) and similar techniques which have been
used by the authors of \cite{SeWeZh-21} to study the spectral gap of spherical triangles. The reader new to analytic perturbation theory may also find \cite{ElSoufiIlias} as a useful reference giving a similar application of this theory.

\subsection{The set of spherical triangles and the associated spectral problem}
\label{sec:defns}

Let $A_*$ and $B_*$ be two points at distance $\frac{\pi}{2}$ on the unit sphere in $\R^3$.
We choose one of the two hemispheres that have $A_*$ and $B_*$ on its boundary and denote by $\Mcal$ the set of triangles
whose vertices are $A_*$, $B_*$ and $C$, where $C$ is any point of that hemisphere. For any $T$ in $\Mcal$, we denote by $a$ the length of
the side opposite to $A_*$ (resp.\ $b$ and $c$) any by $\alpha$ the angle at $A_*$ (resp.\ $\beta$ at $B_*$ and $\gamma$ at $C$).
Figure~\ref{fig:ex1} summarizes these notations.
\begin{figure}
\begin{tikzpicture}[line cap=round,line join=round,scale=2,%
                    x={({-\xx cm,-\xy cm})},y={(\xx cm,-\xy cm)},z={(0 cm,\zy cm)}]
\def\atheta{0}
\def\aphi  {89}
\def\btheta{90}
\def\bphi  {89}
\def\ctheta{45}
\def\cphi  {1}
\pgfmathsetmacro\ax{cos(\atheta)*sin(\aphi)}
\pgfmathsetmacro\ay{sin(\atheta)*sin(\aphi)}
\pgfmathsetmacro\az{cos(\aphi)});
\pgfmathsetmacro\bx{cos(\btheta)*sin(\bphi)}
\pgfmathsetmacro\by{sin(\btheta)*sin(\bphi)}
\pgfmathsetmacro\bz{cos(\bphi)});
\pgfmathsetmacro\cx{cos(\ctheta)*sin(\cphi)}
\pgfmathsetmacro\cy{sin(\ctheta)*sin(\cphi)}
\pgfmathsetmacro\cz{cos(\cphi)});
\pgfmathsetmacro\px{crossx(\ax,\ay,\az,\bx,\by,\bz)}
\pgfmathsetmacro\py{crossy(\ax,\ay,\az,\bx,\by,\bz)}
\pgfmathsetmacro\pz{crossz(\ax,\ay,\az,\bx,\by,\bz)}
\pgfmathsetmacro\qx{crossx(\cx,\cy,\cz,\ax,\ay,\az)}
\pgfmathsetmacro\qy{crossy(\cx,\cy,\cz,\ax,\ay,\az)}
\pgfmathsetmacro\qz{crossz(\cx,\cy,\cz,\ax,\ay,\az)}
\pgfmathsetmacro\rx{crossx(\bx,\by,\bz,\cx,\cy,\cz)}
\pgfmathsetmacro\ry{crossy(\bx,\by,\bz,\cx,\cy,\cz)}
\pgfmathsetmacro\rz{crossz(\bx,\by,\bz,\cx,\cy,\cz)}
\greatcircle{\px}{\py}{\pz}{blue}{-1}{1}
\greatcircle{\qx}{\qy}{\qz}{blue}{ 1}{1}
\greatcircle{\rx}{\ry}{\rz}{blue}{-1}{1}
\draw (0,0,0) circle (1 cm);
\draw[gray,dashed] (0,0,0) -- (1,0,0);
\draw[gray,dashed] (0,0,0) -- (0,1,0);
\draw[gray,dashed] (0,0,0) -- (0,0,1);
\draw[gray,-latex] (1,0,0) -- (1.5,0,0) node (X) [left]  {$x$};
\draw[gray,-latex] (0,1,0) -- (0,1.5,0) node (Y) [right] {$y$};
\draw[gray,-latex] (0,0,1) -- (0,0,1.5) node (Z) [above] {$z$};
\fill[blue] (\ax,\ay,\az) circle (0.8pt);
\fill[blue] (\ax,\ay,\az-0.08) node [below right]  {\small$A_*$};
\fill[blue] (\ax,\ay,\az) node [above right]  {\small$\frac{\pi}{2}$};
\fill[blue] (\bx,\by,\bz) circle (0.8pt);
\fill[blue] (\bx,\by,\bz-0.08) node [below left] {\small$B_*$};
\fill[blue] (\bx,\by,\bz) node [above left] {\small$\frac{\pi}{2}$};
\fill[blue] (\cx-0.01,\cy,\cz) circle (0.8pt) node [above]  {\small$C_*$};
\fill[blue] (\cx,\cy,\cz) node [below]  {\small $\frac{\pi}{2}$};
\fill (0,0,-0.69)  node [above]  {\small$c$};
\fill (0,0.68,0.4)  node [above]  {\small$a$};
\fill (0.68,0,0.4)  node [above]  {\small$b$};
\end{tikzpicture}
\begin{tikzpicture}[line cap=round,line join=round,scale=2,%
                    x={({-\xx cm,-\xy cm})},y={(\xx cm,-\xy cm)},z={(0 cm,\zy cm)}]
\def\atheta{0}
\def\aphi  {89}
\def\btheta{90}
\def\bphi  {89}
\def\ctheta{55}
\def\cphi  {60}
\pgfmathsetmacro\ax{cos(\atheta)*sin(\aphi)}
\pgfmathsetmacro\ay{sin(\atheta)*sin(\aphi)}
\pgfmathsetmacro\az{cos(\aphi)});
\pgfmathsetmacro\bx{cos(\btheta)*sin(\bphi)}
\pgfmathsetmacro\by{sin(\btheta)*sin(\bphi)}
\pgfmathsetmacro\bz{cos(\bphi)});
\pgfmathsetmacro\cx{cos(\ctheta)*sin(\cphi)}
\pgfmathsetmacro\cy{sin(\ctheta)*sin(\cphi)}
\pgfmathsetmacro\cz{cos(\cphi)});
\pgfmathsetmacro\px{crossx(\ax,\ay,\az,\bx,\by,\bz)}
\pgfmathsetmacro\py{crossy(\ax,\ay,\az,\bx,\by,\bz)}
\pgfmathsetmacro\pz{crossz(\ax,\ay,\az,\bx,\by,\bz)}
\pgfmathsetmacro\qx{crossx(\cx,\cy,\cz,\ax,\ay,\az)}
\pgfmathsetmacro\qy{crossy(\cx,\cy,\cz,\ax,\ay,\az)}
\pgfmathsetmacro\qz{crossz(\cx,\cy,\cz,\ax,\ay,\az)}
\pgfmathsetmacro\rx{crossx(\bx,\by,\bz,\cx,\cy,\cz)}
\pgfmathsetmacro\ry{crossy(\bx,\by,\bz,\cx,\cy,\cz)}
\pgfmathsetmacro\rz{crossz(\bx,\by,\bz,\cx,\cy,\cz)}
\greatcircle{\px}{\py}{\pz}{blue}{-1}{1}
\greatcircle{\qx}{\qy}{\qz}{blue}{1}{-1}
\greatcircle{\rx}{\ry}{\rz}{blue}{-1}{1}
\draw (0,0,0) circle (1 cm);
\draw[gray,dashed] (0,0,0) -- (1,0,0);
\draw[gray,dashed] (0,0,0) -- (0,1,0);
\draw[gray,dashed] (0,0,0) -- (0,0,1);
\draw[gray,-latex] (1,0,0) -- (1.5,0,0) node (X) [left]  {$x$};
\draw[gray,-latex] (0,1,0) -- (0,1.5,0) node (Y) [right] {$y$};
\draw[gray,-latex] (0,0,1) -- (0,0,1.5) node (Z) [above] {$z$};
\fill[blue] (\ax,\ay,\az) circle (0.8pt) node [below]  {\small$A_*$};
\fill[blue] (\ax-0.08,\ay,\az-0.06) node [right]  {\small$\alpha$};
\fill[blue] (\bx,\by,\bz) circle (0.8pt) node [below] {\small$B_*$};
\fill[blue] (\bx,\by+0.07,\bz) node [above left] {\small$\beta$};
\fill[blue] (0,1.22,0.98) circle (0.8pt) node [above]  {\small$C$};
\fill[blue] (0,1.27,0.98) node [below left]  {\small $\gamma$};
\fill (0,0,-0.69)  node [above]  {\small$c$};
\fill (0,1.08,0.35)  node [above]  {\small$a$};
\fill (-0.2,0.15,-0.1)  node [above]  {\small$b$};
\end{tikzpicture}
\caption{On the left: the equirectangle triangle $T_*=T(\frac{\pi}{2},\frac{\pi}{2})$. On the right: a generic triangle $T(\alpha,\beta)$ in $\mathcal M$.}
\label{fig:ex1}
\end{figure}
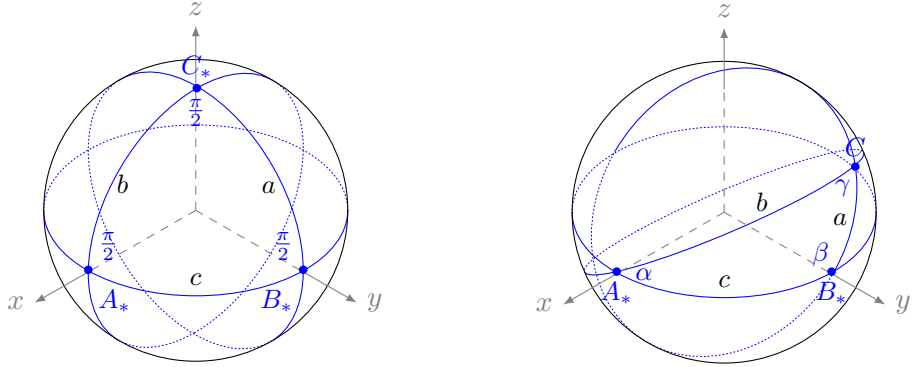

\begin{rem}
  Strictly speaking, to properly define the \textit{set of triangles with one side of length $\pid$} we should
  mod out by the involution $(\alpha,\beta) \leftrightarrow (\beta,\alpha)$. We do not need this subtlety here and
  may freely work on $\Mcal$.
\end{rem}

The set $\Mcal$ is naturally parametrized by $(\alpha,\beta) \in (0,\pi)\times (0,\pi)$ and we will denote by $T(\alpha,\beta)$ the
corresponding triangle. Analyticity on $\Mcal$ means analyticity in $(\alpha,\beta)$.    

We also define the distance between two triangles $T$ and $T'$ by 
\begin{equation*}
   d(T,T')= \max (|\alpha-\alpha'|,|\beta-\beta'|).
\end{equation*}
We let $T_*= T(\frac{\pi}{2},\frac{\pi}{2})$ and $A_*,B_*,C_*$ its vertices.

For any fixed $\beta$, when $\alpha$ goes to $0$, the triangle $T(\alpha,\beta)$ degenerates onto the arc $A_*B_*$, and when $\alpha$
goes to $\pi$, it degenerates onto $\Dcal_\beta$: the digon (or spherical lune) of opening angle $\beta$, see Figure~\ref{fig:ex2} (left).

We will use (spherical) polar coordinates at $A_*$: the point $M(r,\theta)$ is at distance
$r$ along the geodesic that emanates from $A_*$, making the angle $\theta$ with the arc $A_*B_*$; see Figure~\ref{fig:ex2} (right).

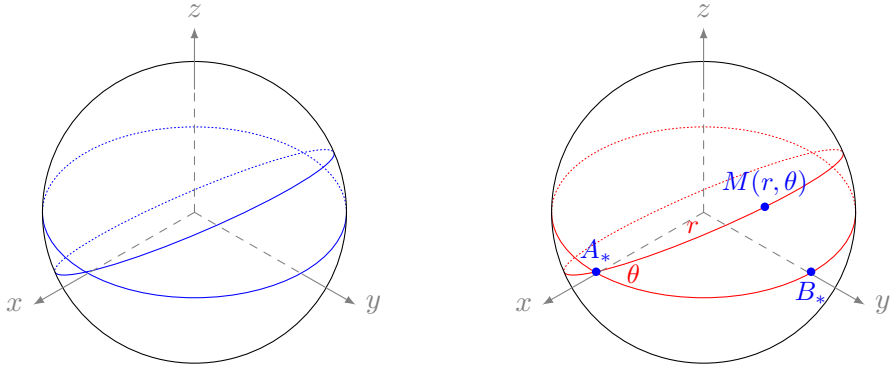
\begin{figure}
\begin{tikzpicture}[line cap=round,line join=round,scale=2,%
                    x={({-\xx cm,-\xy cm})},y={(\xx cm,-\xy cm)},z={(0 cm,\zy cm)}]
\def\atheta{0}
\def\aphi  {89}
\def\btheta{90}
\def\bphi  {89}
\def\ctheta{55}
\def\cphi  {60}
\pgfmathsetmacro\ax{cos(\atheta)*sin(\aphi)}
\pgfmathsetmacro\ay{sin(\atheta)*sin(\aphi)}
\pgfmathsetmacro\az{cos(\aphi)});
\pgfmathsetmacro\bx{cos(\btheta)*sin(\bphi)}
\pgfmathsetmacro\by{sin(\btheta)*sin(\bphi)}
\pgfmathsetmacro\bz{cos(\bphi)});
\pgfmathsetmacro\cx{cos(\ctheta)*sin(\cphi)}
\pgfmathsetmacro\cy{sin(\ctheta)*sin(\cphi)}
\pgfmathsetmacro\cz{cos(\cphi)});
\pgfmathsetmacro\px{crossx(\ax,\ay,\az,\bx,\by,\bz)}
\pgfmathsetmacro\py{crossy(\ax,\ay,\az,\bx,\by,\bz)}
\pgfmathsetmacro\pz{crossz(\ax,\ay,\az,\bx,\by,\bz)}
\pgfmathsetmacro\qx{crossx(\cx,\cy,\cz,\ax,\ay,\az)}
\pgfmathsetmacro\qy{crossy(\cx,\cy,\cz,\ax,\ay,\az)}
\pgfmathsetmacro\qz{crossz(\cx,\cy,\cz,\ax,\ay,\az)}
\pgfmathsetmacro\rx{crossx(\bx,\by,\bz,\cx,\cy,\cz)}
\pgfmathsetmacro\ry{crossy(\bx,\by,\bz,\cx,\cy,\cz)}
\pgfmathsetmacro\rz{crossz(\bx,\by,\bz,\cx,\cy,\cz)}
\greatcircle{\px}{\py}{\pz}{blue}{-1}{1}
\greatcircle{\qx}{\qy}{\qz}{blue}{ 1}{-1}
\draw (0,0,0) circle (1 cm);
\draw[gray,dashed] (0,0,0) -- (1,0,0);
\draw[gray,dashed] (0,0,0) -- (0,1,0);
\draw[gray,dashed] (0,0,0) -- (0,0,1);
\draw[gray,-latex] (1,0,0) -- (1.5,0,0) node (X) [left]  {$x$};
\draw[gray,-latex] (0,1,0) -- (0,1.5,0) node (Y) [right] {$y$};
\draw[gray,-latex] (0,0,1) -- (0,0,1.5) node (Z) [above] {$z$};
\end{tikzpicture}
\begin{tikzpicture}[line cap=round,line join=round,scale=2,%
                    x={({-\xx cm,-\xy cm})},y={(\xx cm,-\xy cm)},z={(0 cm,\zy cm)}]
\def\atheta{0}
\def\aphi  {89}
\def\btheta{90}
\def\bphi  {89}
\def\ctheta{55}
\def\cphi  {60}
\pgfmathsetmacro\ax{cos(\atheta)*sin(\aphi)}
\pgfmathsetmacro\ay{sin(\atheta)*sin(\aphi)}
\pgfmathsetmacro\az{cos(\aphi)});
\pgfmathsetmacro\bx{cos(\btheta)*sin(\bphi)}
\pgfmathsetmacro\by{sin(\btheta)*sin(\bphi)}
\pgfmathsetmacro\bz{cos(\bphi)});
\pgfmathsetmacro\cx{cos(\ctheta)*sin(\cphi)}
\pgfmathsetmacro\cy{sin(\ctheta)*sin(\cphi)}
\pgfmathsetmacro\cz{cos(\cphi)});
\pgfmathsetmacro\px{crossx(\ax,\ay,\az,\bx,\by,\bz)}
\pgfmathsetmacro\py{crossy(\ax,\ay,\az,\bx,\by,\bz)}
\pgfmathsetmacro\pz{crossz(\ax,\ay,\az,\bx,\by,\bz)}
\pgfmathsetmacro\qx{crossx(\cx,\cy,\cz,\ax,\ay,\az)}
\pgfmathsetmacro\qy{crossy(\cx,\cy,\cz,\ax,\ay,\az)}
\pgfmathsetmacro\qz{crossz(\cx,\cy,\cz,\ax,\ay,\az)}
\pgfmathsetmacro\rx{crossx(\bx,\by,\bz,\cx,\cy,\cz)}
\pgfmathsetmacro\ry{crossy(\bx,\by,\bz,\cx,\cy,\cz)}
\pgfmathsetmacro\rz{crossz(\bx,\by,\bz,\cx,\cy,\cz)}
\greatcircle{\px}{\py}{\pz}{red}{-1}{1}
\greatcircle{\qx}{\qy}{\qz}{red}{ 1}{-1}
\draw (0,0,0) circle (1 cm);
\draw[gray,dashed] (0,0,0) -- (1,0,0);
\draw[gray,dashed] (0,0,0) -- (0,1,0);
\draw[gray,dashed] (0,0,0) -- (0,0,1);
\draw[gray,-latex] (1,0,0) -- (1.5,0,0) node (X) [left]  {$x$};
\draw[gray,-latex] (0,1,0) -- (0,1.5,0) node (Y) [right] {$y$};
\draw[gray,-latex] (0,0,1) -- (0,0,1.5) node (Z) [above] {$z$};
\fill[blue] (\ax,\ay,\az) circle (0.8pt) node [above]  {\small$A_*$};
\fill[red] (0.81,0,-0.09) node [right]  {\small$\theta$};
\fill[blue] (\bx,\by,\bz) circle (0.8pt) node [below] {\small$B_*$};
\fill[blue] (-0.17,0.40,0.16) circle (0.8pt) node [above]  {\small$M(r,\theta)$};
\fill[red] (0.5,0.40,0.2) node [above]  {\small$r$};
\end{tikzpicture}
\caption{On the left: a digon (or spherical lune) is a domain bounded by two great circles. On the right: the spherical coordinates $(r,\theta)$}
\label{fig:ex2}
\end{figure}

The side $[B_*,C]$ is parametrized, in these polar coordinates, by the mapping
$\theta \mapsto L_{\beta}(\theta)$ that is implicitly defined by the following application of the cotangent four-part formula:
(that is simplified using that the distance between $A_*$ and $B_*$ is
$\frac{\pi}{2}$)
\begin{equation*}
 0\,=\,\cot L_{\beta}(\theta) -\cot \beta \sin \theta. 
\end{equation*}
This equation can be solved by setting
\begin{equation*}
\forall \beta,\theta \in (0,\pi),~~L_{\beta}(\theta)\,\,=\,\arccot\left( \cot \beta \sin \theta\right),
\end{equation*}
with $\arccot$ the reciprocal function to $\cot$ with values in $(0,\pi)$.
The mapping $(\beta,\theta) \mapsto L_\beta(\theta)$ is thus analytic on $(0,\pi)^2$ and, for any $\beta \in (0,\pi)$,
the mapping $\theta \mapsto L_\beta(\theta)$ extends smoothly to $\R$.

Thus we have the parametrization:
\begin{equation}
\label{paramT}
    T(\alpha,\beta)\,=\,\big\{ (r,\theta):~~ \theta \in (0,\alpha),~r\in(0,L_{\beta}(\theta)) \big \}.
\end{equation}

In these polar coordinates, the spherical metric reads $g\,=\, dr^2\,+\, \sin^2(r)d\theta^2$,
the area element is $\sin r dr d\theta$ and the Dirichlet energy quadratic form for the triangle $T=T(\alpha,\beta)$ is,
for any $u\in \ C_0^\infty(T)$,
\begin{equation}
\label{eq:Dirichlet_form}
  q(u)\,=\,\int_{T} \Big [ \big | \partial_r u(r,\theta) \big |^2 \,+
  \,\frac{1}{\sin ^2r}\big | \partial_\theta u(r,\theta) \big |^2 \Big ]\,\sin rdrd\theta.    
\end{equation}
We also denote by $n$ the Riemannian $L^2$ norm on $T$:
\begin{equation}
\label{eq:Riemannian_norm}
  n(u)\,=\,\int_T \big | u(r,\theta)\big |^2 \, \sin r drd\theta.
\end{equation}
We will abuse notation by also using $q$ and $n$ to denote the bilinear forms that are canonically associated with $q$ and $n$.

We now explain how to associate a self-adjoint operator (that we call the Dirichlet Laplace operator) to this setting. 
The procedure is quite standard and we refer the reader to \cite{Reed-Simon} for more details.
It is well-known that when $Q$ is a bounded quadratic form on a Hilbert space $\Hcal$ with scalar product $n$, 
there exists a unique associated self-adjoint operator that satisfies 
\[
\forall x\in \Hcal, Q(x)\,=\,n(Ax,x).
\] 
The latter statement can be extended to closed unbounded quadratic forms. However, with the definitions above and since the quadratic form $q$ 
is defined on $C_0^\infty(T)$ only, it is not closed. In order to prove that the quadratic form $q$ is closable, we remark that, 
using integration by parts, there exists a partial differential operator $P$ such that 
\[
\forall u\in C_0^\infty(T),~~q(u)\,=\,n(Pu,u).
\]
Moreover, the operator $P$, with domain $C_0^\infty(T)$ is formally symmetric so that we can use the Friedrichs extension 
procedure. As a result, the Dirichlet Laplace operator is obtained as follows.
We first define $H^1_0(T)$ to be the completion of $C_0^\infty(T)$ with respect to the quadratic form $q+n$.
The quadratic form $q$ with domain $H^1_0$ is now closed and the unique associated self-adjoint operator
is the Dirichlet Laplace operator on $T$. We denote it by $\Delta$ (observe that, by construction, $\Delta$ is a non-negative
operator). Despite the corners, the injection from $H^1_0(T)$ into $L^2(T)$
is still compact so that the spectrum of $\Delta$ consists solely of eigenvalues of finite multiplicity.
The construction implies that a function $u$ is an eigenfunction of $\Delta$ with eigenvalue $\lambda$ if and only if the
following system is satisfied:
\begin{equation}\label{eq:eigen_pb}
  \left \{
    \begin{array}{l}
      u \in H^1_0(T),\\
      \forall v \in H^1_0(T) ,~~q(u,v)\,=\,\lambda n(u,v).
    \end{array}
  \right .
\end{equation}

\begin{rem}
  Elaborating on the results of appendix \ref{app:reg}, it can be proved that the eigenfunctions of the latter eigenvalue problem
  do vanish on the sides of the triangles, hence justifying the ``Dirichlet'' appellation.
\end{rem}
\subsection{Analyticity of the spectrum}
For each triangle in $\Mcal$, the spectral problem \eqref{eq:eigen_pb} gives a spectrum that is usually organized in
a non-decreasing sequence:
\begin{equation*}
   \lambda_1(T)\,<\, \lambda_2(T)\,\leq \, \cdots \leq \, \lambda_n(T)\,\leq \, \cdots
\end{equation*}
Each eigenvalue is repeated according to its multiplicity, and we have used the known fact that the first eigenvalue $\lambda_1(T)$ is simple.

The theory of analytic perturbations gives conditions under which the spectrum of a family of
such spectral problems depends analytically on its parameters. We refer to \cite{Kato} for a complete account on the theory and
we now wish to apply the theory when the parameters $(\alpha,\beta)$ vary. 

Let $T_0= T(\alpha_0,\beta_0)$ be a triangle in $\Mcal$, and let $T=T(\alpha,\beta)$ be another triangle in a small neighbourhood of $T_0$.
We recall that $L_{\beta_0}$ and $L_\beta$ are the functions that are used to describe $T_0$ and $T$ in polar coordinates, see \eqref{paramT}.

Analytic perturbation theory applies to a family of quadratic forms on a fixed Hilbert space.
It cannot be used directly here since the spectral problems associated with $T$ and $T_0$ are not defined in the
same Hilbert space, and the corresponding quadratic forms do not have the same domain. In order to circumvent this problem,
we first define a diffeomorphism between $T_0$ and $T$. We want this diffeomorphism to depend analytically on
$(\alpha, \beta)$, but it is actually not necessary to define very precisely what the latter means: analyticity will be checked on the expression
of the quadratic forms in the end.

In order to get Hadamard variational formulas (which we will obtain in Theorems~\ref{thm:Hadamard_1} and \ref{thm:Hadamard_2}), it is convenient to choose our diffeomorphisms as follows.
We choose $\chi$ to be a smooth non-negative and non-increasing function on $\R$ such that $\chi$ is identically $1$ on $(-\infty,\frac{1}{3})$
and identically $0$ on $(\frac{2}{3},\infty),$ and we fix some $\ep>0$. Let $\Phi$ be the mapping defined on $T_0$ by
\begin{equation*}
   \Phi(r,\theta)\,=\,(R, \Theta),
\end{equation*}
with
\begin{equation*}
\left\{\begin{array}{lcl}
  \Theta(r,\theta)&=&\theta\,+\,(\alpha-\alpha_0)\chi\bigl(\frac{\alpha_0-\theta}{\ep}\bigr),\\
  R(r,\theta)&=&r\,+\,(L_\beta\circ\Theta(r,\theta)-L_{\beta_0}(\theta))\chi\bigl(\frac{L_{\beta_0}(\theta)-r}{\ep}\bigr).
  \end{array}\right.
\end{equation*}
This mapping actually depends on $\alpha$, $\beta$, and $\ep$, i.e., $\Phi=\Phi_{\alpha,\beta}^{(\ep)}$, but for readability, the notation does not reflect it.
We also set $\ell_\beta = L_\beta\circ\Theta$ and $\ell_0= L_{\beta_0}$ and observe that these functions depend only on $\theta$.

We now pull back the spherical metric on $T$ to $T_0$ using this diffeomorphism. We thus introduce the Jacobian matrix of
$\Phi$:
\begin{equation*}
  \Jac \Phi_{|(r,\theta)} \,=\,
  \begin{pmatrix}
    A(r,\theta) & C(r,\theta) \\
    0 & B(r,\theta)
  \end{pmatrix},
\end{equation*}
where we have set:
\begin{equation*}
\left\{\begin{array}{ccccl}
  A(r,\theta)&=&\, \partial_rR(r,\theta)&=&\,1-(\ell_{\beta}(\theta)-\ell_0(\theta))\frac{1}{\ep}\chi'\bigl(\frac{\ell_0(\theta)-r}{\ep}\bigr),\\
  \\
  B(r,\theta)&=&\,\partial_\theta \Theta(r,\theta)&=&\,1-(\alpha-\alpha_0)\frac{1}{\ep}\chi'\bigl(\frac{\alpha_0-\theta}{\ep}\bigr),\\
  \\
  C(r,\theta)&=&\,\partial_{\theta} R(r,\theta)&=&\,(\ell_\beta'(\theta)-\ell_0'(\theta))\chi\bigl(\frac{\ell_0(\theta)-r}{\ep}\bigr)\,+\,(\ell_\beta(\theta)-\ell_0(\theta))\frac{\ell_0'(\theta)}{\eps}\chi'\bigl(\frac{\ell_0(\theta)-r)}{\ep}\bigr).
\end{array}\right.
\end{equation*}
Observe that $\Theta$ does not depend on $r$, so that $\partial_r \Theta(r,\theta)=0$. From these expressions, we derive the following lemma.

\begin{lem}
  For any $\ep>0$ there exists $\rho_\ep$ such that $\Phi$ is a smooth diffeomorphism from $T_0$ onto $T$ as soon as
  $d(T,T_0)<\,\rho_\ep$.
\end{lem}
\begin{proof}
  We first choose $\alpha$ close enough to $\alpha_0$ so that $B$ is uniformly bounded below by some positive (small) constant.
  It follows that $\Theta$ is a smooth diffeomorphism from $[0,\alpha_0]$ onto $[0,\alpha]$ and that $|\Theta(\theta)-\theta|=O(|\alpha-\alpha_0|)$
  uniformly. By definition $(\beta,\theta)\mapsto L_\beta(\theta)$ is smooth so that, if $\rho$ is small enough, then
  $A$ is also bounded below by some positive constant. It follows that $\Phi$ is a smooth bijective mapping from $T_0$ onto $T$.
  Restricting $\rho$ again if needed, we can ensure the Jacobian matrix to be always invertible and this proves the claim.
\end{proof}

Using $\Phi$, $T$ is then parametrized by $T_0$.
The pulled-back metric $\Phi^*(dr^2+\sin^2 rd\theta^2)$ is now represented by the matrix $\Gbb$ defined by
\begin{equation*}
  \Gbb(r,\theta)\,=\,{}^t\Jac \Phi_{|(r,\theta)}
  \begin{pmatrix}
    1 & 0 \\
    0 & \sin^2R(r,\theta) \\
  \end{pmatrix}
  \Jac \Phi_{|(r,\theta)}. 
\end{equation*}
It is convenient to set $D(r,\theta)= \sin R(r,\theta)$ and to define the (Euclidean) gradient 
\begin{equation*}
   \nabla u\,=\,\begin{pmatrix} \partial_ru \\ \partial_\theta u \end{pmatrix}.
\end{equation*}   
With these notations, the Dirichlet quadratic
form \eqref{eq:Dirichlet_form} now reads
\begin{equation}\label{def:qt}
  \begin{array}{rcl}
  q(u\,;\,\alpha, \beta) &=&\displaystyle \int_{T_0} {}^t\nabla u(r,\theta) \Gbb^{-1}(r,\theta) \nabla u(r,\theta)
  \,ABD(r,\theta) dr d\theta\smallskip\\
  &=& \dis \int_{T_0} \left[ (C^2+D^2B^2)(\partial_r u)^2-2AC \partial_ru\partial_\theta u +A^2 (\partial_\theta u)^2\right ]
  \frac{drd\theta}{ABD}
  \end{array}
\end{equation}
and the $L^2$ scalar product \eqref{eq:Riemannian_norm} reads
\begin{equation}\label{def:nt}
  n(u\,;\, \alpha,\beta)\,=\,\int_{T_0} u^2 ABD \, drd\theta.
\end{equation}

Using the definitions, we first observe that the quadratic forms $q( \cdot \,;\, \alpha, \beta)$ are uniformly equivalent for $(\alpha,\beta)$ in a small neighbourhood of $(\alpha_0,\beta_0)$, and similarly for $n(\cdot \,;\,\alpha,\,\beta)$. The completion procedure 
that is used to define the Friedrichs extension thus yields a domain that does not depend
on $(\alpha,\beta)$ and thus coincides with $H^1_0(T_0)$.

Moreover, for any fixed $u\in H^1_0(T_0)$, the functions $(\alpha,\beta)\mapsto q(u\,;\,\alpha, \beta)$ and
$(\alpha,\beta)\mapsto n(u\,;\,\alpha,\beta)$ are analytic for $(\alpha,\beta)$ close to $(\alpha_0,\beta_0)$.
It follows that analytic perturbation theory applies and yields the following properties:
\begin{itemize}
\item If $\lambda_0$ is a simple eigenvalue of $T_0$, then there exists $\delta>0$ and a neighbourhood of $T_0$, such that,
  in this neighbourhood, there is a unique eigenvalue of $T$ in $(\lambda_0-\delta, \lambda_0+\delta)$ and this eigenvalue depends
  analytically on $(\alpha,\beta)$.
\item For any (real-)analytic curve $t\mapsto (\alpha(t),\beta(t))$ on some interval $I$, there exists a collection
  $\bigl( t\mapsto E_i(t)\bigr)_{i\geq 1}$ of real-analytic functions that exhaust the spectrum of $T_t= T(\alpha(t),\beta(t))$.
  Such a function is called an analytic eigenvalue branch
  and there also exist corresponding analytic eigenfunction branches $t\mapsto u_i(t)$.
\item The derivatives of the eigenbranches are given by the Feynman-Hellmann formula (see \cite{Kato} or \cite{HJTAMS} prop. 4.6 for a proof in a similar setting):
  for an analytic eigenbranch $t\mapsto (E(t),u(t))$,
  we have 
  \begin{equation}
  \label{eq:Feynman-Hellmann}
    \forall t\in I,~~\dot{E}(t)\|u(t)\|^2\,=\,\left[\dot{q}_t-E(t)\dot{n}_t\right]\left(u(t)\right),
  \end{equation}
  in which the dot denotes the derivative with respect to $t$. This formula is obtained by differentiating \eqref{eq:eigen_pb};
 specifically, we first differentiate $q_t(v)$ and $n_t(v)$ with a fixed $v$ and then evaluate $v=u$.
\end{itemize}

If $\lambda_0$ is an eigenvalue of $T_0$ of multiplicity $m$, it follows by standard min-max arguments that, for $\delta$ small enough,
there exist exactly $m$ eigenvalues of $T$ in $(\lambda_0-\delta,\lambda_0+\delta)$ in a small neighbourhood of $T_0$.
In a nutshell, analytic perturbation theory says that, along
any curve that is real-analytic, it is possible to label these $m$ eigenvalues so as to have analytic functions. There are, however, two problems remaining.
First, the labeling does not preserve the order of eigenvalues: analytic eigenbranches will typically cross at $T_0$. Then, it
is usually not possible to define eigenbranches that would be analytic for $(\alpha,\beta)$ in a neighbourhood: the labeling depends on
the analytic curve that is chosen and cannot be done consistently in all directions. Of course both problems only arise for multiple eigenvalues.

We have expressed the derivatives of the eigenvalue branch using the corresponding eigenfunction branch. For the reasons given in the preceding paragraph, it is convenient to give a way to
recover the derivatives without knowing a priori the eigenfunction branch. This is obtained by the following procedure.

Let $\lambda_0$ be an eigenvalue of $T_0$ and $\Ecal_0$ the corresponding eigenspace.
The derivatives of all the eigenbranches that coincide with $\lambda_0$ at $t=0$ are exactly the eigenvalues
of the quadratic form $\dot{q}-\lambda_0\dot{n}$, restricted to $\Ecal_0$ and relative to the scalar product $n$.
Observe that using \eqref{paramT} we can write
\begin{equation}
\label{eq:der}
  \dot{q}-\lambda_0\dot{n}\,=\,\dot{\alpha}\bigl( \partial_\alpha q-\lambda_0 \partial_\alpha n\bigr )
  \,+\,\dot{\beta}\bigl( \partial_\beta q-\lambda_0 \partial_\beta n\bigr),
\end{equation}
so that, although we may not have differentiability of the eigenvalues, still, it is enough to know the
partial derivatives $\partial_\alpha q-\lambda_0 \partial_\alpha n$ and $\partial_\beta q-\lambda_0 \partial_\beta n$
to compute the derivatives of the eigenbranches in any direction.

\subsection{A Hadamard variational formula}
The formulas in the preceding section express the derivative of the eigenbranches using integrals over
the whole domain $T_0$, of some quadratic expressions in $u,\,\partial_r u,\,\partial_\theta u$, see \eqref{eq:Feynman-Hellmann} and \eqref{eq:der}.  
Hadamard variational formulas use integrals only on the boundary of the domain, see Theorems~\ref{thm:Hadamard_1} and \ref{thm:Hadamard_2} below. Since the latter are of independent
interest and give
slightly simpler computations in the end, we explain here how to derive them. This derivation is made possible by computing $\dot{q}-\lambda_0\dot{n}$ for fixed $\ep$ and then letting
our parameter $\ep$ go to $0$. More precisely, for any $\ep$, we define the two quadratic forms (see \eqref{eq:der}) 
\begin{equation*}
  D_\alp^\ep=\partial_\alpha q-\lambda_0 \partial_\alpha n~\text{and}~D_\beta^\ep=\partial_\beta q-\lambda_0 \partial_\beta n,
\end{equation*}
that are obtained from \eqref{def:qt} and \eqref{def:nt},
where recall that the dependence on $\ep$ comes from the diffeomorphism $\Phi_{\alpha,\beta}^{(\ep)}$.

\begin{prop}
\label{prop:first_limits}
  Let $\lambda_0$ be an eigenvalue of $T_0$ and $\Ecal_0$ the corresponding eigenspace. For any $u\in \Ecal_0$,
  \begin{equation*}
    \begin{split}
      \lim_{\ep \rightarrow 0} D_\alp^\ep(u)&=\, -\int_0^{\ell(\alpha)} \frac{|\partial_\theta u(r,\alpha)|^2}{\sin r} \, dr,\\
      \lim_{\ep \rightarrow 0} D_\beta^\ep(u)&=-\int_0^{\alpha_0} \left [ |\partial_r u(L_{\beta_0}(\theta),\theta)|^2
        \,+\,\frac{|\partial_\theta u(L_{\beta_0}(\theta),\theta)|^2}{\sin^2 L_{\beta_0}(\theta)}\right ]
    (\partial_\beta L)_{\beta_0}(\theta) \sin L_{\beta_0}(\theta) \, d\theta.\\
    \end{split}
  \end{equation*}
\end{prop}
\begin{proof}
  Since $u\in H^1_0$, the expressions  $D_\alp^\ep(u)$ and  $D_\beta^\ep(u)$ can be obtained by differentiating under the integral sign the expressions
  given in \eqref{def:qt} and \eqref{def:nt}.

  Thus, for $D_\alp^\ep$, we need to compute $F(r,\theta\,;\,\alpha_0,\beta_0),\,\partial_\alpha F(r,\theta\,;\,\alpha_0,\beta_0)$
  for $F=A,B,C,D$. Of these four quantities, only $\partial_\alpha B$ does not vanish identically and after a somewhat lengthy but straightforward
  computation, we obtain
  \begin{equation*}
    D_\alp^{\ep}(u)\,=\, \int_{T_0} \frac{1}{\ep} \chi'\Bigl(\frac{\alpha_0-\theta}{\ep}\Bigr) \left[ \frac{|\partial_\theta u|^2}{\sin^2r} \,+\,
     \lambda_0 |u|^2 \right ]\sin r dr d\theta.
  \end{equation*}
  We now let $\ep$ go to $0$. When tested against sufficiently well-behaved functions, $\frac{1}{\ep}\chi'(\alpha_0-\theta)$ converges to the
  integration on the side
\begin{equation*}
   \big\{ (r,\alpha_0):~~r\in [0,L_{\beta_0}(\alpha_0)]\big \}.
\end{equation*}
We will provide, in Appendix~\ref{app:application_Hadamard},
  all the necessary estimates showing that this limit is justified when $u$ is an eigenfunction. We then obtain  
  \begin{equation*}
    \lim_{\ep \rightarrow 0} D_\alp^\ep(u)=\, -\int_0^{\ell(\alpha)} \frac{|\partial_\theta u(r,\alpha)|^2}{\sin r} \, dr.
  \end{equation*}
  The second term vanishes since $u$ satisfies the Dirichlet boundary condition.

  For $D_\beta^\ep$ we follow the same strategy, computing now the derivatives with respect to $\beta,$ still evaluated at $(\alpha_0,\beta_0)$.
  We find:
  \begin{equation*}
    \begin{split}
      D_\beta^\ep(u)\,=\,&\dis \int_{T_0} \partial_\beta L_{\beta_0}|\partial_r u|^2\left[ (\cos r)
        \chi\Bigl(\frac{L_{\beta_0}(\theta)-r}{\ep}\Bigr)\,+\,\frac{\sin r}{\ep}\chi'\Bigl(\frac{L_{\beta_0}(\theta)-r}{\ep}\Bigr)\right ]\, drd\theta\\
      & -2 \dis \int_{T_0} \frac{\partial_r u \partial_\theta u}{\sin r} \left [ \bigl(\partial_\beta L'_{\beta_0}\bigr)\chi\Bigl(\frac{L_{\beta_0}(\theta)-r}{\ep}\Bigr)
        \,+\,\partial_\beta L_{\beta_0}\frac{L'_{\beta_0}}{\ep} \chi'\Bigl(\frac{L_{\beta_0}(\theta)-r}{\ep}\Bigr) \right ]\, dr d\theta\\
      & -\dis \int_{\beta_0} \frac{|\partial_\theta u|^2}{\sin^2 r}\left[ (\cos r)
        \chi\Bigl(\frac{L_{\beta_0}(\theta)-r}{\ep}\Bigr)\,+\,\frac{\sin r}{\ep}\chi'\Bigl(\frac{L_{\beta_0}(\theta)-r}{\ep}\Bigr)\right ]\, drd\theta \\
      &-\lambda_0 \dis \int_{T_0} \partial_\beta L_{\beta_0} |u|^2 \left[ (\cos r)\chi\Bigl(\frac{L_{\beta_0}(\theta)-r}{\ep}\Bigr)
        -\frac{\sin r}{\ep}\chi'\Bigl(\frac{L_{\beta_0}(\theta)-r}{\ep}\Bigr)\right ]\, drd\theta.
    \end{split}
  \end{equation*}  
  As above, we will give in Appendix~\ref{app:application_Hadamard} the needed estimates to prove that the terms with  $\chi\bigl(\frac{L_{\beta_0}(\theta)-r}{\ep}\bigr)$
  converge to $0$ and the terms with $\frac{1}{ \ep}\chi'\bigl(\frac{L_{\beta_0}(\theta)-r}{\ep}\bigr)$ converge to a boundary
  integral over the side $\big\{ (L_{\beta_0}(\theta),\theta):\,\theta\in [0,\alp_0]\big \}$. Since $u$ satisfies the Dirichlet boundary condition,
  $u$ vanishes on the latter side. Denoting by $\gamma$ the parametrization $\theta \mapsto (L_{\beta_0}(\theta),\theta)$, 
  we obtain 
    \begin{multline*}
    \lim_{\ep \rightarrow 0}  D_\beta^\ep(u)=\,\\
     -\int_0^{\alpha_0} \left[|\partial_r u\circ \gamma |^2
    -2 L_{\beta_0}'(\theta)\frac{(\partial_r u\partial_\theta u)\circ \gamma }{\sin^2 L_{\beta_0}(\theta)}
    -\frac{|\partial_\beta u\circ \gamma|^2}{\sin^2 L_{\beta_0}(\theta)}\right ] (\partial_{\beta}L)_{\beta_0}(\theta)
  \sin (L_{\beta_0}(\theta)) \, d\theta.
\end{multline*}
  This expression can be simplified further by observing that Dirichlet boundary condition implies that
  $ \left ( L_{\beta_0}'\partial_r u \,+\,\partial_\theta u\right)\circ \gamma\,=\,0.$
  Finally, we obtain
  \begin{equation*}
   \lim_{\ep \rightarrow 0} D_\beta^\ep(u)=-\int_0^{\alpha_0} \left [ |\partial_r u\circ \gamma|^2
        \,+\,\frac{\partial_\theta u\circ \gamma|^2}{\sin^2 L_{\beta_0}(\theta)}\right ]
    \partial_\beta L(\theta\,;\alpha_0,\beta_0) \sin L_{\beta_0}(\theta) \, d\theta.\qedhere
  \end{equation*}
\end{proof}

\begin{rem}
\label{rem:cont}
  The estimates of the appendix are needed to properly prove the convergence when $\ep$ goes to zero for any
  eigenvalue branch at any triangle $T_0$. The proof of Theorem~\ref{thm:main_formulation_spectral} needs
  this computation only for the first three eigenvalues and at the equirectangle triangle. In the latter case, we have 
  an explicit expression for the eigenfunction so that the convergence can be proved directly without referring to
  the general Sobolev theory on singular domains.   
\end{rem}

Combining the latter proposition and the results of analytic perturbation
theory that we have recalled in the previous
section, we obtain the following two theorems.

\begin{thm}
\label{thm:Hadamard_1}
  Let $T_0 \in \Mcal$ and $\lambda_0$ be a simple eigenvalue of the Dirichlet spherical Laplace operator 
  of $T_0$. There exist $\delta>0$ and a neighbourhood $U\subset \Mcal$ of $T_0$ such that:
  \begin{itemize}
  \item Any triangle $T$ in $U$ has a unique eigenvalue $\lambda$ in $(\lambda_0-\delta,\lambda_0+\delta)$.
  \item The mapping $T \mapsto \lambda(T)$ is real-analytic on $U$.
  \item For any $T= T(\alpha,\beta)\in U$, we have
    \begin{equation*}
    \begin{split}
      \partial_\alpha\lambda(T)&=\, -\int_0^{L_{\beta}(\alpha)} \frac{|\partial_\theta u(r,\alpha)|^2}{\sin r} \, dr,\\
      \partial_\beta\lambda(T)&=-\int_0^{\alpha} \left [ |\partial_r u(L_{\beta}(\theta),\theta)|^2
        \,+\,\frac{ | \partial_\theta u(L_{\beta}(\theta),\theta)|^2}{\sin^2 L_{\beta}(\theta)}\right ]
    (\partial_\beta L)_{\beta}(\theta) \sin L_{\beta}(\theta) \, d\theta,\\
    \end{split}
    \end{equation*}
    where $u$ is a $L^2(T)$ normalized eigenfunction.  
  \end{itemize}
\end{thm}

\begin{thm}
\label{thm:Hadamard_2}
  For $t$ in an interval $I$, let $t\mapsto T(\alpha(t),\beta(t))= T_t$ be an analytic family of spherical triangles and $\lambda_0$ be an eigenvalue
  of multiplicity $m$ of $T_0$. Then there exist $m$ analytic functions $(E_k)_{1\leq k\leq m}$ defined on $I$ such that:
  \begin{enumerate}
  \item There exist $\delta_0, t_0>0$ such that, for any $t\in (-t_0,t_0)$ and any eigenvalue $\lambda$
    in $\spec (T_t)\cap (\lambda_0-\delta_0,\lambda_0+\delta_0)$, the multiplicity of $\lambda$ is the number of $k$ such that
    $E_k(t)=\lambda$.
  \item The derivatives $\dot{E}_k(0)$ are the eigenvalues of the quadratic form
    \begin{align}\label{def:deriv}
        u \mapsto &-\dot{\alpha}(0)\int_0^{L_{\beta}(\alpha)} \frac{|\partial_\theta u(r,\alpha)|^2}{\sin r} \, dr\\
         &-\dot{\beta}(0) \int_0^{\alpha} \left [ |\partial_r u(L_{\beta}(\theta),\theta)|^2
        \,+\,\frac{\partial_\theta u(L_{\beta}(\theta),\theta)|^2}{\sin^2 L_{\beta}(\theta)}\right ]
      (\partial_\beta L)_{\beta}(\theta) \sin L_{\beta}(\theta) \, d\theta,\nonumber
    \end{align}
    restricted to the eigenspaces of $\lambda_0$ and relatively to the $L^2$ norm on $T$.
    \end{enumerate}
\end{thm}

\begin{proof}
The proof of the two theorems follows the same line. First we fix some $\ep>0$. The first statements, in particular the existence of $U$ 
(or $t_0$), follow from the previous section using the family of diffeomorphisms
$\Phi^\ep$. It remains to compute the derivatives. For this, we pick a triangle $T$ in the neighbourhood $U$, and we write,
for each derivative and each $\ep>0$ the formula that is obtained using $\Phi^\ep$ (with now $T$ as the starting point).
Using analyticity, the eigenvalue branches that we obtain do not depend on $\ep$. So for each $\ep$, the formula
for the derivative of simple eigenbranches gives the same value. For a multiple eigenvalue, the quadratic form that gives the derivatives has
the same eigenvalues. We can thus let $\ep$ go to zero and and Proposition~\ref{prop:first_limits} then yields the result.
\end{proof}

\subsection{The first eigenvalue on $\Mcal$}

In this section, we describe the first eigenvalue $\lambda_1$ as a function on the set of spherical triangles
$\Mcal$.

It is well known that the first eigenvalue of a domain in a Riemannian manifold is always simple. It then follows
that the eigenvalue $\lambda_1$ depends analytically on $(\alpha,\beta)$ in $(0,\pi)\times (0,\pi)$. We now make a
list of several facts that help us understand the level sets of $\lambda_1$.
\begin{enumerate}
  \item\textit{Symmetry:} the symmetry with respect to the median hyperplane of $[A_*,B_*]$ in the sphere 
    exchanges $T(\alpha,\beta)$ and $T(\beta,\alpha)$. The function $\lambda_1$ is thus symmetric with respect to
    $(\alpha,\beta)\leftrightarrow (\beta,\alpha)$.
  \item\textit{Monotonicity:} if $\alpha'\geq \alpha$ and $\beta'\geq \beta$, then the triangle $T(\alpha,\beta)$ is a subset of
    $T(\alpha',\beta')$. Using the min-max principle, Dirichlet eigenvalues are shown to be decreasing relative to the inclusion of domains.
    We thus infer:
    \begin{equation*}   \alpha'\geq \alpha~~\text{and}~~\beta'\geq \beta~~\implies~~\lambda_1(\alpha',\beta') \leq \lambda_1(\alpha,\beta).
    \end{equation*}
    Since the first Dirichlet eigenvalue of the hemisphere is $2$, we also get that
    \begin{equation*}
      \forall T\,\in \Mcal,~\lambda_1(T)\geq 2.
    \end{equation*}
  \item\textit{Regularity:}
    \begin{equation*}
      \forall (\alpha,\beta)\in (0,\pi)\times (0,\pi),~~\partial_\alpha \lambda_1(\alpha,\beta)<0,~~\text{and}~~
      \partial_\beta \lambda_1(\alpha,\beta)<0.
    \end{equation*}
    Let us prove that  $\partial_\alpha \lambda_1\,\neq\,0$ by contradiction. If this derivative vanishes then
    the integral formula of Theorem~\ref{thm:Hadamard_1} implies that $\partial_\theta u$
    vanishes on one side of the triangle. If we reflect the triangle across this side, we obtain a rhombus to which
    we extend $u$ by $0$. We denote by $\widetilde{u}$ this extension and we test $(\Delta-\lambda)\widetilde{u}$ against a smooth function with compact support
    in the rhombus. Using integration by parts (Green's formula) inside and outside the original triangle, we obtain an integral over
    the side across which we have reflected. This integral vanishes because $u$ and
    $\partial_\theta u$ vanish on that side. This proves that $\widetilde{u}$ is an eigenfunction (with the same eigenvalue)
    of the Dirichlet Laplace operator in the rhombus. This violates the principle of analytic continuation for eigenfunctions.
    By symmetry, the derivative with respect to $\beta$ cannot vanish. The monotonicity
    gives the sign.
  \item\textit{Behaviour near the boundary:}
    If $\alpha$ or $\beta$ goes to $0$, then $\lambda_1$ goes to infinity. This is a general fact about shrinking domains
    with Dirichlet boundary condition. For instance, here, we could use the minmax principle to compare with a spherical angular sector 
    whose angle goes to $0$.
    
    If $\alpha$ goes to $\pi$ and $\beta$ goes to $\beta_0\in (0,\pi]$, the first eigenvalue
    $\lambda_1(\alpha,\beta)$ converges to the first eigenvalue of the $\Dcal_{\beta_0}$, the
    digon of opening angle $\beta_0$ (see Figure~\ref{fig:ex2}, left). Indeed, the family of spectral problems is continuous up to 
$(0,\pi]\times (0,\pi]$.
\item\textit{Behaviour on the boundary:} All the computations we made are still valid for $\alpha=\pi$ and varying $\beta$.
  It follows that the mapping $\mu$ which, to an angle $\beta \in (0,\pi]$, assigns
  the first Dirichlet eigenvalue of the digon of opening angle $\beta$, is an analytic, decreasing diffeomorphism from
  $(0,\pi)$ onto $(2,\infty)$ that extends continuously at $\pi$.
\end{enumerate}
\begin{rem}
\label{rem:digon}
    Observe that in \cite{SeWeZh-21} (see also \cite{Wa-74,WaKe-77}), it is proved that
  the spectrum of the digon of angle $\beta$ can be explicitly computed. This computation yields
  $\mu(\beta)\,=\,\frac{\pi}{\beta}(\frac{\pi}{\beta}+1)$ which obviously satisfies all claimed features.
\end{rem}
These properties allow us to prove the following proposition, which gives a reasonably complete understanding of how $\lambda_1$
behaves as a function on $\Mcal$. See Figure~\ref{fig:path_triangles} for an illustration.

\begin{prop}
  For any $c\in (2,\infty)$, there exists $\alpha_c\in (0,\pi)$ such that the first eigenvalue of $\Dcal_{\alpha_c}$ is $c$.
  The level set $\lambda_1^{-1}\{c\}\subset \Mcal$ is an analytic curve that can be globally parametrized by
  $\alpha\in (\alpha_c,\pi)$. More precisely, there exists an analytic function $\Bcal_c$ from $(\alpha_c,\pi)$ such that
  \begin{equation*}
    \lambda_1(T(\alpha,\beta))\,=\, c ~~\iff~~\beta \,=\,\Bcal_c(\alpha).
  \end{equation*}
  The function $\Bcal_c$ is decreasing, extends continuously to $[\alpha_c,\pi]$ by $\Bcal_c(\alpha_c)=\pi$ and $\Bcal_c(\pi)=\alpha_c$
  and $\Bcal_c'(\frac{\pi}{2})\,=\,-1$.
\end{prop}

\begin{figure}
\begin{center}
\includegraphics[width=.69in, trim=3.5cm 1cm 4cm 1.5cm, clip]{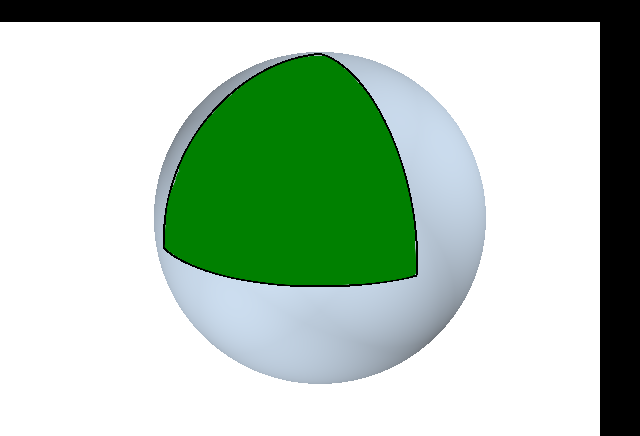}
\includegraphics[width=.69in, trim=3.5cm 1cm 4cm 1.5cm, clip]{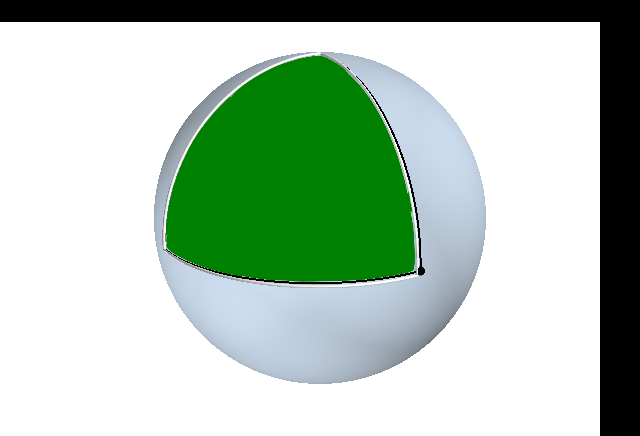}
\includegraphics[width=.69in, trim=3.5cm 1cm 4cm 1.5cm, clip]{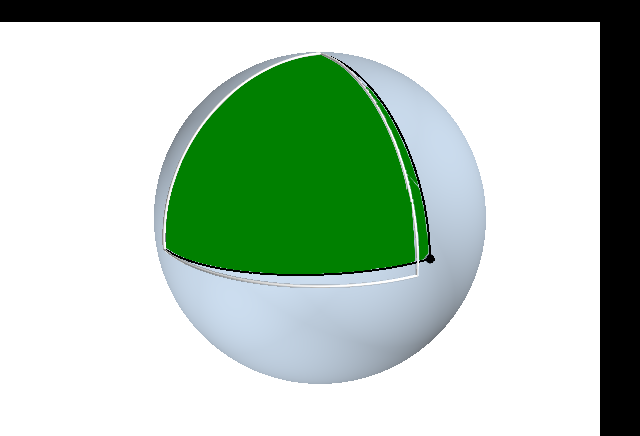}
\includegraphics[width=.69in, trim=3.5cm 1cm 4cm 1.5cm, clip]{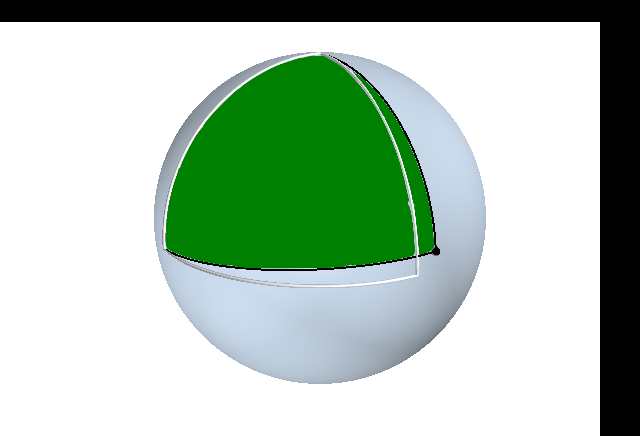}
\includegraphics[width=.69in, trim=3.5cm 1cm 4cm 1.5cm, clip]{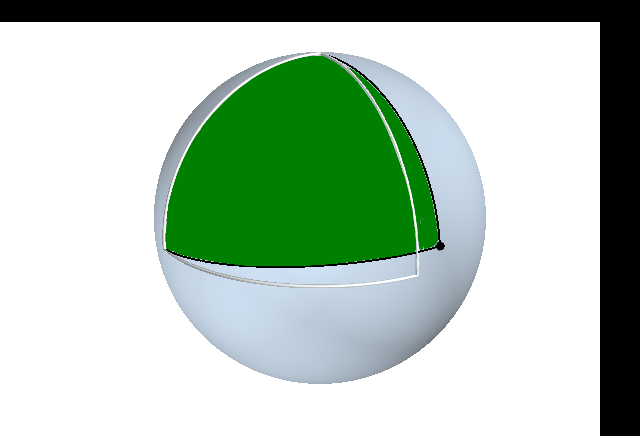}
\includegraphics[width=.69in, trim=3.5cm 1cm 4cm 1.5cm, clip]{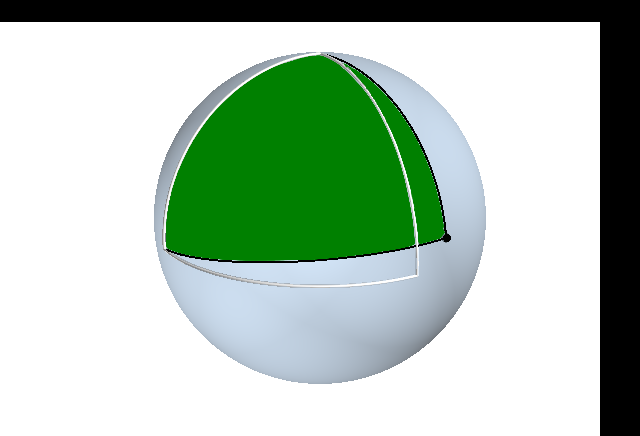}
\includegraphics[width=.69in, trim=3.5cm 1cm 4cm 1.5cm, clip]{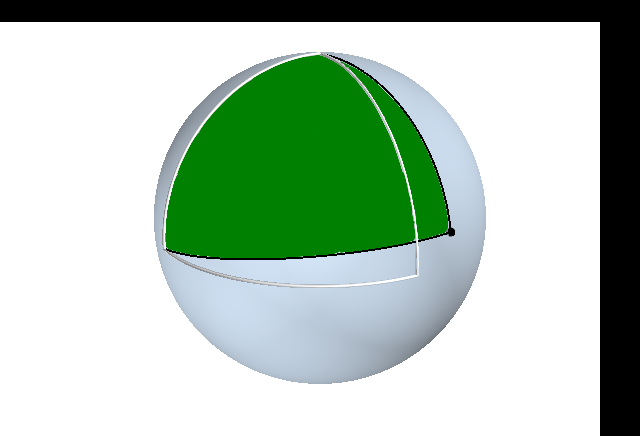}
\includegraphics[width=.69in, trim=3.5cm 1cm 4cm 1.5cm, clip]{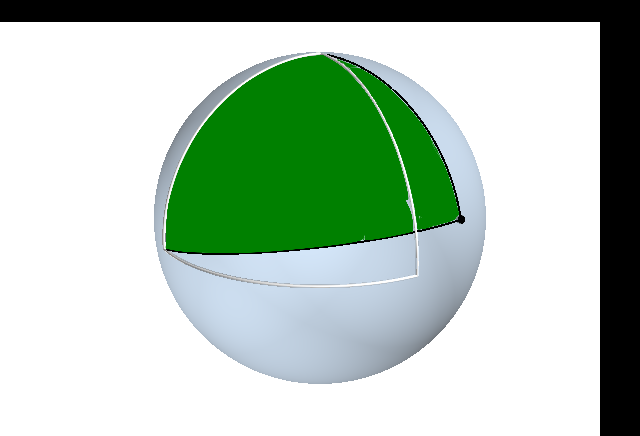}

\includegraphics[width=.69in, trim=3.5cm 1cm 4cm 1.5cm, clip]{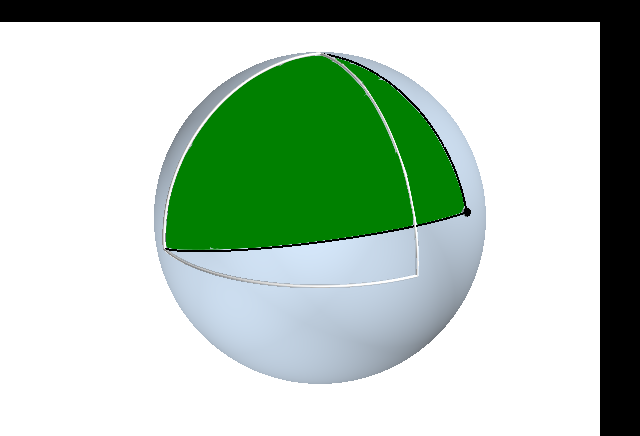}
\includegraphics[width=.69in, trim=3.5cm 1cm 4cm 1.5cm, clip]{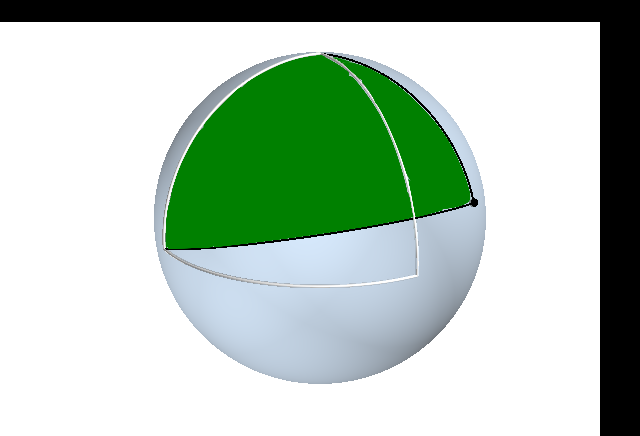}
\includegraphics[width=.69in, trim=3.5cm 1cm 4cm 1.5cm, clip]{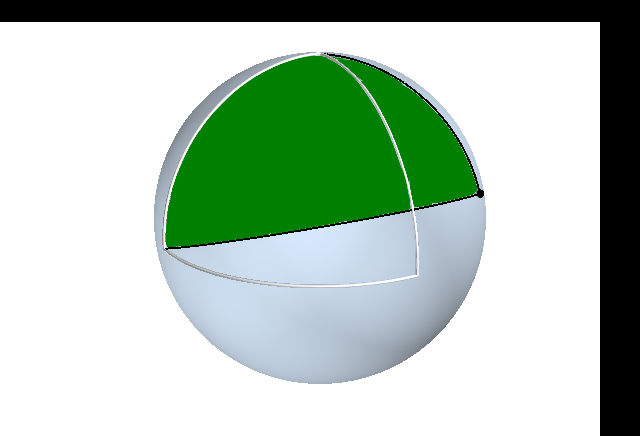}
\includegraphics[width=.69in, trim=3.5cm 1cm 4cm 1.5cm, clip]{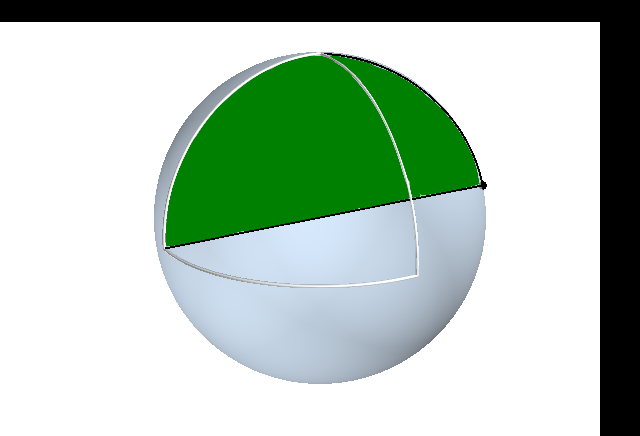}
\includegraphics[width=.69in, trim=3.5cm 1cm 4cm 1.5cm, clip]{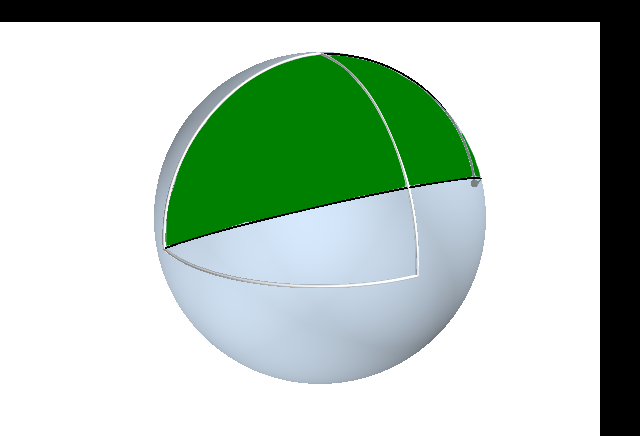}
\includegraphics[width=.69in, trim=3.5cm 1cm 4cm 1.5cm, clip]{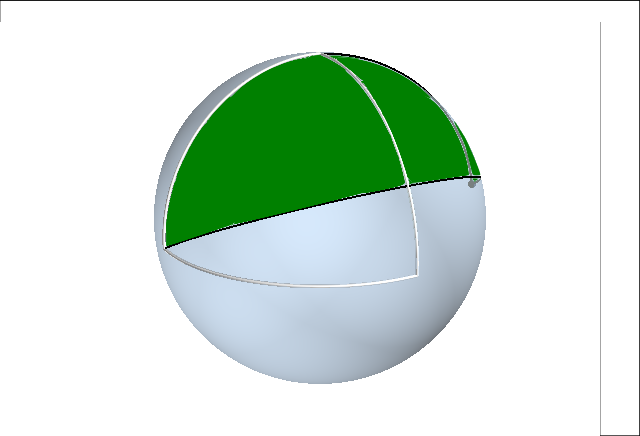}
\includegraphics[width=.69in, trim=3.5cm 1cm 4cm 1.5cm, clip]{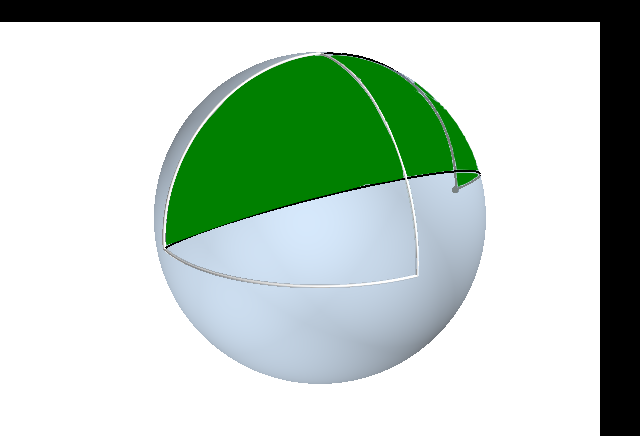}
\includegraphics[width=.69in, trim=3.5cm 1cm 4cm 1.5cm, clip]{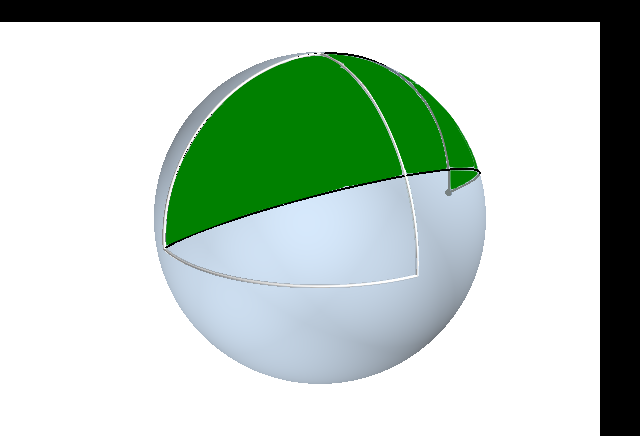}
\end{center}
\caption{All these triangles have first Dirichlet eigenvalue $\lambda_1=12$. The first domain is the equirectangle triangle $T_*$ and the last domain is the digon of opening $\frac{\pi}{3}$. The second eigenvalue is strictly decreasing along this path of triangles, going from $30$ to $20$.}
\label{fig:path_triangles}
\end{figure}

\begin{proof}
  The first statement follows from the known behaviour of the first eigenvalue of digons, see Remark~\ref{rem:digon}. Moreover,
  using monotonicity in $\beta$, for each $\alpha> \alpha_c$ the mapping $\beta \mapsto \lambda_1(\alpha,\beta)$
  is a decreasing diffeomorphism onto $(\mu(\alpha),\infty)$. Since this interval contains $c$, there is a unique $\beta$
  that we denote by $\Bcal_c(\alpha)$ such that $\lambda_1(\alpha,\beta)=c$.
  The fact that $\Bcal_c(\alpha)$ is analytic then follows from the implicit function theorem.
  The remaining statements follow from the behaviour of $\lambda_1$ near and on the boundary and the symmetry.
\end{proof}

\section{Proof of Theorem \ref{thm:main_formulation_spectral}}
\begin{proof}
Let $T_*$ be the equirectangle triangle with vertices $A_*,\,B_*,\,C_*$, see Figure~\ref{fig:ex1}. It corresponds to $\alpha=\beta=\frac{\pi}{2}$.
We have $\lambda_1(T_0) =12$. In the sequel, we set $\Bcal=\Bcal_c$ for $c=12$.
Using the formula in \cite[Proposition 1.1]{SeWeZh-21} for the eigenvalues of the digons,
we obtain that the function $\Bcal$ is defined on $[\frac{\pi}{3},\pi]$, analytic in $(\frac{\pi}{3},\pi)$ and
\begin{equation*}
  \forall \alp \in \Bigl[\frac{\pi}{3},\pi\Bigr],~~\lambda_1(\alp, \Bcal(\alp))\,=\,12.
\end{equation*}
Let $I$ be a small interval around $0$, for $t\in I$, we define
\begin{equation*}
  \alpha(t)\,=\,\frac{\pi}{2}+t,~~\beta(t)\,=\,\Bcal\circ \alpha(t)
\end{equation*}
and set $T_t\,=\,T(\alpha_t,\beta_t)$. The curve $t\mapsto T_t$ is real analytic and, using the properties of $\Bcal$,
we know that the spectrum of $T_t$ is organized into analytic branches. Since the second eigenvalue of $T_0$ is of multiplicity
$2$, we can restrict $I$ so that there exists three analytic functions $E_1,E_2,E_3$ such that
\begin{equation*}
  \forall t\in I,~\lambda_1(t)\,=\,E_1(t)\,=12,~~\lambda_2(t)\,=\,\min\{E_2(t),E_3(t)\},~\lambda_3(t)\,=\,\max\{ E_2(t),E_3(t)\}.
\end{equation*}

Let $\Ecal$ be the (two-dimensional) eigenspace corresponding to $\lambda_2(0)$.
We make the formulas \eqref{def:deriv} explicit, using in particular that $(\partial_\beta L)_{\frac{\pi}{2}}(\theta)\,=\sin \theta$ and
that $\dot{\alp}(0)=1$ and $\dot{\beta}(0)=-1$. Thus, the derivatives of $E_2$ and $E_3$ are given by the
eigenvalues of the quadratic form 
\begin{equation*}
  d(u)\,=\,-\int_0^{\frac{\pi}{2}} \frac{|\partial _\theta u(r,\frac{\pi}{2})|^2}{\sin r} \, dr +
  \int_0^{\frac{\pi}{2}} |\partial_r u(\textstyle{\frac{\pi}{2}},\theta)|^2\sin \theta \, d\theta
\end{equation*}
with respect to the $L^2$ norm.
We define on $T_0$ the functions $u_1$ and $u_2$ by 
\begin{equation*}
\left\{\begin{array}{rcl}
u_1(r,\theta)&=&\sqrt{\frac{1155}{8\pi}}\left( 3\cos^5r-4\cos^3r+\cos r\right)\sin 2\theta,\\
u_2(r,\theta)&=&\sqrt{\frac{3465}{32\pi}}\cos r \sin^4 r \sin 4\theta,
\end{array}\right.
\end{equation*}  
that form a $L^2$-orthonormal basis of $\Ecal$ (see in \cite[Cor.~2.1]{SeWeZh-21} for instance). 

We compute
\begin{equation*}\dis
  \begin{array}{rclcl}
    \dis \int_{0}^{\frac{\pi}{2}} \frac{|\partial_\theta u_1(r,\frac{\pi}{2})|^2}{\sin r}\, dr
    &=&\dis \frac{1155}{8\pi}\int_0^{\pid} \frac{4\left(3\cos^5 r -4\cos^3r+\cos r\right)^2}{\sin r}\, dr
    &=&\dis \frac{44}{\pi},\smallskip\\
    \dis \int_{0}^{\frac{\pi}{2}} \frac{|\partial_\theta u_2(r,\frac{\pi}{2})|^2}{\sin r}\, dr
    &=& \dis \frac{3465}{32\pi}\int_0^{\pid} \frac{16\cos^2 r\sin ^8 r}{\sin r}\, dr
    &=& \dis \frac{44}{\pi},\smallskip\\
    \dis \int_{0}^{\frac{\pi}{2}} \frac{\partial_\theta u_1(r,\frac{\pi}{2})\partial_\theta u_2(r,\frac{\pi}{2})}{\sin r}\, dr
    &=& C \dis \int_0^{\pid} \frac{\left(3\cos^5 r -4\cos^3r+\cos r\right)\cos r \sin^4 r}{\sin r}\, dr
    &=& 0,
  \end{array}
\end{equation*}
where $C$ is some numerical constant that we do not need to write down explicitly, and
\begin{equation*}\dis
  \begin{array}{rclcl}
   \dis  \int_{0}^{\frac{\pi}{2}} |\partial_r u_1(\textstyle{\frac{\pi}{2}},\theta)|^2\sin \theta\, d\theta
    &=& \dis\frac{1155}{8\pi}\int_0^{\pid} \sin^2(2\theta)\sin \theta d\theta
    &=& \dis \frac{77}{\pi},\smallskip\\
    \dis \int_{0}^{\frac{\pi}{2}} |\partial_r u_2(\textstyle{\frac{\pi}{2}},\theta)|^2\sin \theta\, d\theta
    &=& \dis \frac{3465}{32\pi}\int_0^{\pid} \sin^2(4\theta)\sin \theta\, d\theta
    &=& \dis \frac{55}{\pi},\smallskip\\
    \dis \int_{0}^{\frac{\pi}{2}} \partial_r u_1(\textstyle{\frac{\pi}{2}},\theta )\partial_r u_2(\frac{\pi}{2},\theta)\sin \theta\, d\theta
    &=& \dis \frac{1155\sqrt{3}}{16\pi}\int_0^{\pid} \sin (2\theta)\sin(4\theta)\sin \theta\, d\theta
    &=& -11\sqrt{3}.
  \end{array}
\end{equation*}
We obtain that the matrix that represents $d$ in the basis $(u_1,u_2)$ of $\Ecal$ is
\begin{equation*}
  D\,=\,-\frac{11}{\pi}
  \begin{pmatrix}
    3 & -\sqrt{3} \\
    -\sqrt{3} & -3
  \end{pmatrix}.
\end{equation*}
Since $(u_1,u_2)$ is a $L^2$ orthonormal basis of $\Ecal$, the derivatives $\dot{E_2}$ and $\dot{E_3}$ are the eigenvalues of
$D$ that we compute to be $\pm \frac{22\sqrt{3}}{\pi}$.
Finally, we obtain, for $t$ close to $0$,
\begin{equation*}
\left\{\begin{array}{rcl}
  E_2(t)&=&12\,+\,t\cdot \frac{22\sqrt{3}}{\pi}\,+\, o(t),\\
  E_3(t)&=&12\,-\,t\cdot \frac{22\sqrt{3}}{\pi}\,+\, o(t),\\
  \lambda_2(t)&=&30\,-\,|t|\cdot \frac{22\sqrt{3}}{\pi}\,+\, o(t),\\
  \lambda_3(t)&=&30\,+\,|t|\cdot \frac{22\sqrt{3}}{\pi}\,+\, o(t),
\end{array}\right.
\end{equation*}
and the proof is complete.
\end{proof}

\begin{rem}
 Using the computations above, the quadratic form 
that is given by equation~\eqref{def:deriv} at the equirectangle 
triangle when expressed in the basis $(u_1,u_2)$ yields the same expression as 
(3.11) in \cite{SeWeZh-21}. We could have use analyticity to claim that the formula \textit{loc.\ cit.}, 
which is proved for variations of triangles of diameter $\frac{\pi}{2}$ 
(using different variational ---Feynman-Hellmann type---formulas), still holds for any variation. 
We have found it interesting to write down the Hadamard variational formula so as to have a (slightly) 
different proof. 
\end{rem}

\begin{rem}
  This perturbation approach can be implemented starting from any initial triangle $T_*$, and, basically, we only need to show that
  the ratio $\frac{\lambda_2}{\lambda_1}$ cannot be always rational. Although this seems a rather weak statement, our proof requires some
  rather precise knowledge of the eigenfunctions. This explains our choice of the equirectangle triangle.
\end{rem}

\appendix

\section{Regularity of eigenfunctions of a triangle}\label{app:reg}
The aim of this appendix is to provide the necessary estimates that allow to pass to the limit $\ep \rightarrow 0$ in order
to obtain the Hadamard variational formulas. All the results can be extracted from the literature on elliptic
problems in domains with corners (see \cite{Grisvard, Kondratiev} for instance). We have chosen to give some ideas
of the proof so as to have a self-contained presentation.

We fix a spherical triangle $T(\alpha, \beta)$ for some $(\alpha, \beta) \in (0,\pi)^2$. Explicitly, it is the
domain as in \eqref{paramT}
equipped with the metric $g\,=\,dr^2\,+\,\sin^2 r d\theta^2$.
We may see $T$ as a subset of the plane $\R^2$ equipped with the coordinates $x=r\cos \theta$ and $y=r\sin \theta$.
The metric $g$ is then uniformly equivalent to the Euclidean metric $dx^2+dy^2$.

We recall the definition of the usual Sobolev spaces, for $k\in \N$ and using the common multiindex notation:
\begin{equation*}
  \begin{split}
    H^k(T)&=\, \{ u\in L^2(dxdy):~~\forall \alpha,~|\alpha| \leq k,~\partial^\alpha u \in L^2(dxdy) \},\\
    \forall u\in H^k(T),~&~\| u\|_{H^k}^2\,=\,\sum_{|\alpha|\leq k} \| \partial^\alpha u\|_{L^2(dxdy)}.
  \end{split}
\end{equation*}
We also recall that $H^1_0(T)$ is the completion of $C^\infty_0(T)$ with respect to the $H^1$ norm.

Writing the partial derivatives $\partial_x$ and $\partial_y$ in polar coordinates, we see that 
the Sobolev space $H^1(T)$ can alternatively be defined as follows:
\begin{equation*}
  H^1(T)\,=\,\{ u\in L^2(T, \sin r dr d\theta):~~\partial_r u \text{ and } \frac{1}{\sin r}\partial_\theta u \in L^2(T, \sin r dr d\theta)\}.
\end{equation*}

We observe that the set $H^1_0(T)$ is defined in Section~\ref{sec:defns} as the completion of $C^\infty_0(T)$ with respect to
the Dirichlet energy quadratic form $q$. We see here that it coincides with the usual one.

The norm that is associated with the spherical Laplace operator is
\begin{equation*}
  \forall u\in C^\infty_0(T),~~\| u\|_{\Delta}^2\,=\, \|u\|^2_{L^2(\sin r dr d\theta)}\,+\,\| \Delta u\|^2_{L^2(\sin r dr d\theta)}. 
\end{equation*}
By standard ellipticity estimates, this norm is equivalent to the $H^2$ norm.
\begin{lem}
The space $C^\infty_0(T)$ is dense in $\dom \Delta$ with respect to the $\|\cdot \|_{\Delta}$ norm. 
\end{lem}

\begin{proof}
  Let $u\in \dom \Delta$ that is orthogonal to $C^\infty_0(T)$ with respect to the Hilbertian norm $\|\cdot \|_{\Delta}$.
  Double $T$ to form a surface $\widetilde{T}$ that is a sphere equipped with a singular spherical metric that has three
  cone points $A,B,C$ at which the total angle is (strictly) less
  than $2\pi$. Extend $u$ a function $\widetilde{u}$ on $\widetilde{T}$ that is odd with respect to the natural involution of $\widetilde{T}$.
  Using Green's formula, by construction, we have  
  \begin{equation*}
    \forall \phi \in C^\infty_0(\widetilde{T}\setminus \{A,B,C\}),~~\langle \widetilde{u}, (\Delta+1) \phi \rangle \,=\,0.
  \end{equation*}
  It follows that, in the distributional sense, $(\Delta +1)\widetilde{u}$ is supported at
  $A,B,C$. The latter can be explicitly described using polar coordinates and it turns out that, since the total angle
  is less than $2\pi$, $\widetilde{u}$ must also be even. Thus, the function $\widetilde{u}$ vanishes and thus also $u$. 
\end{proof}

The latter lemma gives $\dom \Delta = \mathring{W}^2_2(T)$ in the notations of \cite{Grisvard}, and Theorem~1.4.4.4 \textit{loc.\ cit.}\
applies to give the following corollary.
\begin{cor}
  For any $u \in \dom \Delta$, $(\sin r)^{-2}u$, $(\sin r)^{-1}\partial_r u$ and $(\sin r)^{-2} \partial_\theta u$ belong to
  $L^2(\sin r dr d\theta)$.
\end{cor}

\begin{proof}
Theorem 1.4.4.4 of \cite{Grisvard} gives $r^{-2}u$, $r^{-1}\partial_xu$ and $r^{-1}\partial_y u$ in $L^2(T,\,dxdy)$. The claim follows using
  \begin{equation}\label{eq:firstorder}
      r\partial_r u =\, x\partial_x u \,+\, y\partial_y u \quad \text{and} \quad 
      \partial_\theta u = x\partial_y u -y\partial_x u.  \qedhere
  \end{equation}
\end{proof}

Finally, we obtain:
\begin{prop}
  Let $u$ be an element of the domain of the Dirichlet Laplace operator. Then
  $u$, $\partial_r u$ and $\frac{1}{\sin r}\partial_\theta u $ belong to $H^1(T)$.
\end{prop}
\begin{proof}
  Since $u\in \dom \Delta$, then all derivatives with respect to $(x,y)$ of order at most $2$ are in $L^2$.
  We compute $\partial_r^2 u$ and $\frac{1}{\sin r}\partial_\theta\partial_ru$ in Cartesian coordinates
  using the relations \eqref{eq:firstorder}, and we check by inspection that these are $L^2$.
  We do the same for $\partial_r \frac{1}{\sin r}\partial_\theta u$ and $\frac{1}{\sin^2 r}\partial^2_\theta u$.
\end{proof}

\section{Application to the Hadamard formula}
\label{app:application_Hadamard}
In this section, we use the regularity of eigenfunctions to prove that the limit $\ep \rightarrow 0$ that is used
in the proof of Hadamard variational formula is justified (see the proof of Proposition~\ref{prop:first_limits}). As in Remark~\ref{rem:cont}, we emphasize here that it is actually not needed
to prove Theorem~\ref{thm:main_formulation_spectral} since  we would only need to consider the first three eigenfunctions of
the equirectangle triangle for which explicit bounds can be given.

We begin by addressing the limit of $D_\beta^\ep(u)$ since it gives simpler computations; the global strategy being the same.
We first observe that in all terms, the $\chi$ or $\chi'$ factor cuts-off near $r=L_{\beta_0}(r)$. Using that $u\in H^1_0(T)$,
all terms with the factor $\chi$ converge to $0$ by integration of a $L^1$ function on a shrinking domain. It remains to address
the terms with a $\chi'$. All these terms can be written under the expression
\begin{equation*}
  \B_\ep(h)\,=\, \int_0^{\alpha_0}\int_0^{\ell(\theta)} \frac{1}{\ep}\chi'\Bigl(\frac{L_{\beta_0}(\theta)-r)}{\ep}\Bigr)h(r,\theta) \, dr d\theta, 
\end{equation*}
in which $h$ is obtained as a product of a smooth function (away of $r=0$) times two first order derivatives of $u$.
Since $u\in \dom T$, the partial derivatives $\partial_r u$ and $\partial_\theta u$
are in $H^1(T\cap \{ r\geq r_0>0\})$. It follows that $h$ belongs to $W^1_1(T)$: the set of $L^1$ functions whose gradient is also $L^1$.

Changing variables, we have
\begin{equation*}
  \B_\ep(h)\,=\, \int_0^1 \int_0^{\alpha_0} \chi'(\rho)h(L_{\beta_0}(\theta)-\ep \rho, \theta) d\rho d\theta.
\end{equation*}
Since $\displaystyle \int \chi'(\rho) d\rho = -1$, we see that
\begin{equation*}
  \B_\ep(h)+\int_0^{\alpha_0} h(L_\beta(\theta),\theta)\,d\theta\,=\,
  \int_0^1 \int_0^{\alpha_0} \chi'(\rho)\left[h(L_{\beta_0}(\theta)-\ep \rho, \theta)-h(L_{\beta_0}(\theta),\theta)\right] d\rho d\theta. 
\end{equation*}

We define the mapping $\Phi$ on $[0,1]\times [0,1]\times [0,\alpha_0]$ by $(\rho,s,\theta) \mapsto (\rho, L_{\beta_0}(\theta)-s\ep \rho, \theta)$.
For any fixed $\theta$ and $\rho $, $s\mapsto \Phi(\rho,s,\theta)$  parametrizes a curve
in $T$ and since $u\in W^1_1$, we can write
\begin{equation*}
  h(L_{\beta_0}(\theta)-\ep \rho, \theta)-h(L_{\beta_0}(\theta),\theta)\,=\, -\int_0^1 \ep \rho \partial_rh \circ \Phi(\rho, s,\theta)\, ds.  
\end{equation*}
This leads us to consider the integral
\begin{equation*}
R_\ep\,=\,\int_0^1\int_0^1\int_0^{\alpha_0} \left | \chi'(\rho)\ep \rho \partial_rh \circ \Phi(\rho, s,\theta)\right | dsd\rho d\theta.
\end{equation*}
The mapping $\Phi$ is obviously smooth and injective, and since its Jacobian determinant
is $-\ep \rho$, the change of variables is legitimate. Using that $\rho'$ has support in $[\frac{1}{3},\frac{2}{3}]$, we obtain
\begin{equation*}
  R_\ep \,=\, \int_{U_\ep} \left | \chi'(\rho) \partial_rh (r,\theta)\right | d\rho dr d\theta,
\end{equation*}
where $U_\ep$ is the image of $[\frac{1}{3},\frac{2}{3}]\times [0,1] \times [0,\alpha]$ under $\Phi$. We define the set $V_\ep\subset T$ such that 
\begin{equation*}
  \{1\}\times V_\ep =\Phi \{ 1\}\times [0,1]\times [0,\alpha_0],
\end{equation*}
it is straightforward that $U_\ep\subset [0,1]\times V_\ep$ and that $U_\ep$
is a shrinking neighbourhood in $T$ of the side $\{(L(\theta),\theta):~\theta \in [0,\alpha_0] \}$.
Since $\partial_rh \in L^1(T)$, we obtain that $R_\ep$ goes to $0$ with $\ep$ and, hence,
\begin{equation*}
  \lim_{\ep \rightarrow 0} \B_\ep (h)\,=\,-\int_0^{\alpha_0} h(L_\beta(\theta),\theta)\,d\theta. 
\end{equation*}
Applying this result to the different terms in $D_\beta^\ep(u)$ yields the Hadamard formula for $\partial_\beta E$.

We follow the same strategy to study the limit of $D_\alp^\ep(u)$. We have two terms that can be brought to the expression
\begin{equation*}
   \Abb_\ep(h) \,=\, \int_0^{\alpha_0}\int_0^{L_{\beta_0}(\theta)} \frac{1}{\ep}\chi'\Bigl(\frac{\alp_0-\theta}{\ep}\Bigr) h(r,\theta)\, rdrd\theta,
\end{equation*}
with $h\in W^1_1(T)$.
We first perform a change of variables that fixes the domain of integration.
For $(\rho, t)$ in $[0,L_{\beta_0}(\alpha_0)]\times [0,1]$, we define
\begin{equation*}
  \begin{split}
    \theta &= \alpha_0 - \ep t, \\
    \delta(\rho,t)&= \frac{1}{\ep}\left [ L_{\beta_0}(\alpha_0-\ep t)-L_{\beta_0}(\alpha_0)\right ]
    \chi\Bigl(\frac{L_{\beta_0}(\alpha_0)-\rho}{L_{\beta_0}(\alpha_0)}\Bigr),\\
    r&=\rho\,+\,\ep \delta(r,t).
  \end{split}
\end{equation*}
Observe that $\delta$ also depends on $\ep$.

The mapping $\Phi$ is seen to be a diffeomorphism from $[0,L_{\beta_0}(\alpha_0)]\times [0,1]$
onto $W_\ep$, which is a neighbourhood in $T$ of the side $\{ (r,\alpha_0):~r\in [0,L_{\beta_0}]\}$
that shrinks when $\ep$ goes to $0$. Using $\Phi$ as a change of variables, we obtain
\begin{equation*}
  \Abb_\ep(h)\,=\, \int_0^1\int_0^{L_{\beta_0}(\alpha_0)} \chi'(t) h(\rho+\ep \delta,\alpha_0-\ep t)
  \left[ 1+\ep \partial_\rho\delta\right ](\rho+\ep\delta) d\rho dt.
\end{equation*}
We obtain a sum of three terms:
\begin{equation*}
  \begin{split}
    \Abb^0_{\ep}(h)&=\,\dis \int_0^1\int_0^{L_{\beta_0}(\alpha_0)} \chi'(t) h(\rho+\ep \delta,\alpha_0-\ep t)\rho d\rho dt,\\
    R^1_{\ep}(h)&= \dis \int_0^1\int_0^{L_{\beta_0}(\alpha_0)} \chi'(t) h(\rho+\ep \delta,\alpha_0-\ep t)\ep \partial_\rho\delta \rho d\rho dt,\\
    R^2_{\ep}(h)&=\,\dis \int_0^1\int_0^{L_{\beta_0}(\alpha_0)} \chi'(t) h(\rho+\ep \delta,\alpha_0-\ep t)
  \left[ 1+\ep \partial_\rho\delta\right ]\ep\delta) d\rho dt.
  \end{split}
\end{equation*}
Using that $\delta$ is identically $0$ if $\rho < \frac{1}{3}$ and that
$\delta$ and its first-order derivatives are uniformly bounded with respect to $\ep, t, \rho$, by undoing the change of variables
we see that the terms $R^i, \, i=1,2$ are bounded by the integral of $|h|$ over $W_\ep$. Since $h\in L^1$ the latter goes to $0$
and we are left to study the limit of $\Abb^0_\ep$.

We now set
\begin{equation*}
  R^0_\ep(h)\,=\, \Abb^0_{\ep}(h)\,+\,\int_0^{L_{\beta_0}(\alpha_0)} h(r,\alpha_0) r\,dr
\end{equation*}
so that, as above
\begin{equation*}
  R^0_\ep(h)\,=\, \int_0^1\int_0^{L_{\beta_0}(\alpha_0)}\chi'(t)\left[h(\rho+\ep \delta,\alpha_0-\ep t)-h(\rho,\alpha_0)\right]\, \rho d\rho dt.
\end{equation*}
Fix $(\rho,t)\in [0,L_{\beta_0}(\alpha_0)]\times [0,1]$.
The mapping $\gamma(\cdot\,;\,\rho,t)\,:\,s\mapsto (\rho \,+\, s\ep \delta(\rho, st), \alpha_0-\ep st)$
sends the interval $[0,1]$ onto a smooth curve in $T$ that stays within $W_\ep$. Since $h$ is $W^1_1$, we have
\begin{multline*}
    R^0_\ep(h)=\,\dis \int_0^1\int_0^{L_{\beta_0}(\alpha_0)}\int_0^1 \chi'(t)\bigl[\ep(\delta(\rho,st)+st\partial_t\delta(\rho,st))
    \partial_rh\circ \gamma(s\,;\,\rho,t) \\
    -\,\ep t \partial_\theta\circ \gamma(s\,;\,\rho,t)\bigr] \rho dsd\rho dt.
\end{multline*}
We now argue in a similar fashion as before: we define $\Psi$ on $[0,L_{\beta_0}(\alpha_0)]\times [0,1]\times [0,1]$
by $\Psi(\rho,t,s)\,=\,\left(\rho, \gamma(s\,;\,\rho,t)\right)$. We show that it is a legitimate change of variables,
which turns $R^0_\ep$ into an integral that is bounded above by the integral of a $L^1$ function over $W_\ep$.
It thus tends to $0$ and that makes the proof finally complete.

\end{document}